\documentclass[11pt]{amsart}

\usepackage{amsrefs,amstext,amsmath,amsthm,amsfonts,amssymb,enumerate,amscd,cite,color}

\newtheorem{theorem}{Theorem}[section]

\newtheorem{proposition}[theorem]{Proposition}

\newtheorem{lemma}[theorem]{Lemma}
\newtheorem{corollary}[theorem]{Corollary}
\newtheorem{example}[theorem]{Example}
\newtheorem{definition}[theorem]{Definition}
\newtheorem{remark}[theorem]{Remark}

\newcommand{\NN}{\ensuremath{\mathbb{N}}} 
\newcommand{\ZZ}{\ensuremath{\mathbb{Z}}} 
\newcommand{\RR}{\ensuremath{\mathbb{R}}} 
\newcommand{\CC}{\ensuremath{\mathbb{C}}} 

\newcommand{\defeq}{:=} 
\newcommand{\id}{\mathrm{id}} 
\newcommand{\inv}{^{-1}} 
\newcommand{\esssup}{\mathop\mathrm{esssup}} 
\newcommand{\supp}{\mathrm{supp}} 
\newcommand{\lspan}{\mathrm{span}} 
\newcommand{\Aut}{\mathrm{Aut}} 

\newcommand{\dualp}[2]{\left\langle #1 , #2 \right\rangle} 

\newcommand{\Bd}[1][\HH]{\mathcal{B}(#1)} 
\newcommand{\Cpt}[1][\HH]{\mathcal{K}(#1)} 
\newcommand{\btens}{\mathbin{\overline{\otimes}}} 
\newcommand{\Sone}{\mathcal{S}_1} 
\newcommand{\Stwo}{\mathcal{S}_2} 
\newcommand{\Mult}{\mathcal{M}} 
\newcommand{\projtens}{\mathbin{\widehat{\otimes}}} 

\newcommand{\HH}{\ensuremath{\mathcal{H}}} 
\newcommand{\HHH}{\ensuremath{\mathcal{L}}} 
\newcommand{\ip}[2]{\left\langle #1 , #2 \right\rangle} 

\newcommand{\vN}{\mathrm{vN}} 

\newcommand{\cbm}{\mathrm{CB}} 
\newcommand{\cbw}{\mathrm{CB}^\sigma} 
\newcommand{\cb}{\mathrm{cb}} 

\newcommand{\Sch}{\mathfrak{S}} 
\newcommand{\HS}{\mathrm{HS}} 
\newcommand{\HSconv}{\mathfrak{S}_\mathrm{conv}} 
\newcommand{\HScent}{\mathfrak{S}_\mathrm{cent}} 
\newcommand{\Mcb}{\mathrm{M}_\cb} 

\newcommand{\nph}{\varphi} 
\newcommand{\cl}{\mathcal} 

\begin{document}

\title[Central and convolution multipliers]
{Central and convolution Herz--Schur multipliers}

\author[A. McKee]{Andrew McKee}
\address{Faculty of Mathematics, University of Bia\l ystok, Ul.\ Konstantego Cio\l \-kow\-ski\-ego 1M, Bia\l ystok 15-245, Poland}
\email{amckee240@qub.ac.uk}

\author[R. Pourshahami]{Reyhaneh Pourshahami}
\address{Department of Mathematics, Kharazmi University, 50 Taleghani Ave., 15618, Tehran, Iran}
\email{std\_reyhaneh.pourshahami@khu.ac.ir}

\author[I. G. Todorov]{Ivan G. Todorov}
\address{
School of Mathematical Sciences, University of Delaware, 501 Ewing Hall,
Newark, DE 19716, USA, and 
Mathematical Sciences Research Centre,
Queen's University Belfast, Belfast BT7 1NN, United Kingdom}
\email{todorov@udel.edu}

\author[L. Turowska]{Lyudmila Turowska}
\address{Department of Mathematical Sciences, Chalmers University of Technology and the University of Gothenburg, Gothenburg SE-412 96, Sweden}
\email{turowska@chalmers.se}

\date{1 January 2021}

\begin{abstract}
We obtain descriptions of central operator-valued Schur and Herz--Schur multipliers, akin to a 
classical characterisation due to Grothendieck, that reveals a close link between central (linear) multipliers and 
bilinear multipliers into the trace class. 
Restricting to dynamical systems where a locally compact group acts on itself by translation, 
we identify their convolution multipliers as the right completely bounded multipliers, in the sense of 
Junge--Neufang--Ruan, of a canonical quantum group associated with the underlying group. 
We provide characterisations of contractive idempotent operator-valued Schur and Herz--Schur multipliers. 
Exploiting the link between Herz--Schur multipliers and multipliers on transformation
groupoids,  we provide a combinatorial characterisation of groupoid multipliers that are 
contractive and idempotent. 
\end{abstract}

\maketitle

\tableofcontents

\section{Introduction}
\label{sec:intro}

Schur multipliers originated in 
the work of Schur  on the entry-wise (or Hadamard) product of matrices
in the early twentieth century. 
These are complex-valued functions, defined on the cartesian product $X\times Y$ of 
two measure spaces $(X,\mu)$ and $(Y,\nu)$ that give rise to 
completely bounded maps on the space $\cl K$ of all compact operators from $L^2(X,\mu)$ into $L^2(Y,\nu)$, acting by pointwise multiplication on the integral kernels of the operators from the Hilbert-Schmidt class. 
A concrete description of these objects, which has found numerous applications thereafter, 
was given by  Groth\'endieck in his Resum\'{e} \cite{Gro53}.
Since then, Schur multipliers have played a significant role in operator theory, the theory of Banach spaces, the theory of operator spaces, and have been linked to perturbation theory through the concept of double operator integrals
(see \cite{CleMS, lemerdy-tt} and the references therein).

The theory of Herz--Schur, or completely bounded, multipliers of the Fourier algebra of a locally compact group
originated in the work of Herz~\cite{He74}, where they were viewed as a generalisation of Fourier--Stieltjes transforms. 
Similarly to Schur multipliers, Herz--Schur multipliers are complex-valued functions, this time defined on a locally compact group $G$, that give rise to completely bounded maps 
on the reduced $C^*$-algebra $C^*_r(G)$ of $G$, acting by pointwise multiplication 
on its subalgebra $L^1(G)$. 
An important development in the subject were the works of Gilbert and 
of Bo\.{z}ejko and Fendler~\cite{BoFe84}, showing that the Herz--Schur multipliers on the locally compact group $G$ can be 
isometrically identified with the space of all Schur multipliers on $G\times G$ of Toeplitz type.
Haagerup \cite{Haagerup}
pioneered the use of Herz--Schur multipliers to study the approximation properties of operator algebras
(see also \cite{CaHa85}).

Recently, several generalisations of Schur and Herz--Schur multipliers to the `operator-valued' case have appeared: 
B\'{e}dos and Conti~\cite{BeCo16} introduced multipliers of a $C^*$-dynamical system based on 
a Hilbert module version of the Fourier--Stieltjes algebra, and applied these techniques to study $C^*$-crossed products
while, in \cite{mtt}, 
three of the present authors defined
Schur and Herz--Schur multipliers with values in the space of all completely bound\-ed maps on a 
$C^*$-algebra and obtained a version of the Bo\.{z}ejko--Fendler correspondence.
The use of multiplier techniques to study reduced crossed products, following Haagerup's work, 
has been furthered by Skalski and three of the present authors in \cite{mstt}, 
by the first author in \cite{McKee}, and by the first and the fourth authors in \cite{mck-t}.


In this paper we consider special cases of the multipliers defined in \cite{mtt}. 
We define central Schur and Herz--Schur multipliers 
in Definition~\ref{d_ce} and Definition~\ref{de:centralHSmult}, respectively. 
They are associated with completely bounded maps on a $C^*$-algebra $A$ that 
are multiplication operators by elements of the centre of the multiplier algebra of $A$, 
and are one of the most common type of multipliers that arises in specific circumstances. 
A special case of particular importance arises when $A$ is abelian. 
Given a central Herz--Schur multiplier of the 
$C^*$- dynamical system $(A, G,\alpha)$, 
the corresponding completely bounded map on the crossed product is an $A$-bimodule map. Such maps were considered by Dong and Ruan~\cite{DR12} in their study of the Hilbert module Haagerup property of crossed products.  
Exploiting the fact that commutative (unital) $C^*$-algebras are simply algebras of continuous functions on compact topological spaces,
we identify the central Schur and Herz--Schur multipliers 
with scalar-valued functions on three and two variables, respectively. 
This allows us to identify a close link, that seems to have remained unnoticed until now, 
between central multipliers and the bilinear Schur multipliers into the trace class, 
introduced and characterised by Coine, Le Merdy and Sukochev in \cite{CleMS} (see also \cite{lemerdy-tt}).

A $C^*$-dynamical system of particular importance is $(C_0(G),G,\beta)$, where $G$ is a locally compact group, 
$C_0(G)$ is the $C^*$-algebra of all continuous functions on $G$ vanishing at infinity, and $\beta$ is the left translation action
of $G$ on $C_0(G)$. 
The second main class of maps we are concerned with are 
the convolution multipliers of $(C_0(G),G,\beta)$ introduced in \cite{mtt}.
We answer \cite[Question 6.6]{mtt}, identifying the Herz--Schur multipliers 
of the latter dynamical system with the right multipliers of a canonical quantum group associated with $G$; 
in the case where $G$ is abelian, we show that these multipliers coincide with the 
elements of the Fourier--Stieltjes algebra $B(G \times \Gamma)$, where $\Gamma$ is the dual group of $G$.

%
%

Finally, we investigate when the special classes of multipliers considered in this paper give rise to \emph{idempotent} completely bounded maps.
The general study of idempotent Herz--Schur multipliers goes back to Cohen\cite{coh60}, who characterised all idempotent elements of the measure algebra $M(G)$.
In \cite{Hos86}, Host generalised Cohen's characterisation by identifying the
general form of idempotents in $B(G)$, for any locally compact group $G$,
while Katavolos and Paulsen in \cite{KP05}  and Stan in \cite{Stan09} gave characterisations of contractive 
idempotent Schur multipliers and contractive idempotent Herz-Schur multipliers respectively, based on a combinatorial 3-of-4 property.
In this paper, we use the 3-of-4 property 
to obtain characterisations of 
various classes of central idempotent  Schur multipliers and idempotent Herz--Schur multipliers  of dynamical systems. 

The paper is organised as follows. 
Section~\ref{sec:prelims} contains background material, 
including a review of crossed products and multipliers as introduced in \cite{mtt}.
The section also includes some preliminary results that will be needed later. 
In Section~\ref{sec:centmults} we define central Schur $A$-multipliers, and present a characterisation of the central Schur $C_0(Z)$-multipliers, followed by a similar characterisation of central Schur $A$-multipliers for an arbitrary $C^*$-algebra $A$.
After introducing central Herz--Schur multipliers, we characterise the central Herz--Schur $(A, G,\alpha)$-multipliers,
the central Herz--Schur $(C_0(Z), G,\alpha)$-multipliers, as well as their canonical positive cones. 
Convolution multipliers are considered in Section \ref{sec:idemp}, first in the abelian and then in the general case. 
Therein, we also investigate idempotent multipliers within the classes of central and convolution 
multipliers from Section~\ref{sec:centmults} and Section~\ref{sec:convmults}.

%
%
\section{Preliminaries}
\label{sec:prelims}

Throughout this paper, 
we make standing separability assumptions: 
unless otherwise stated, 
we consider only separable $C^*$-algebras, separable Hilbert spaces and second-countable locally compact groups. 
These assumptions allow us to consider multipliers defined on standard measure spaces.
We however note that the results remain valid for the case of discrete spaces with counting measure, 
in which case the separability assumptions can be dropped.

\subsection{General background}
\label{ssec:generalprelims}

\subsubsection{Measure spaces}

We fix for the whole paper standard measure spaces $(X, \mu)$ and $(Y, \nu)$; 
this means that there exist locally compact, metrisable, complete, separable topologies on $X$ and $Y$ (called admissible topologies), with respect to which $\mu$ and $\nu$ are regular Borel $\sigma$-finite measures. 
The direct products $X\times Y$ and $Y\times X$ are equipped with the corresponding product measures. 
We use standard notation for the $L^p$ spaces over $(X, \mu)$ and $(Y, \nu)$ ($p = 1,2,\infty$); we will also consider (not necessarily countable) sets equipped with counting measure, in which case we write $\ell^p(X)$ in place of $L^p(X)$. 

Given a Banach space $B$, the space $L^p(X , B)$ ($p = 1,2$) is the space of (equivalence classes of) 
Bochner $p$-integrable functions from $X$ to $B$ with respect to $\mu$; 
each of these spaces contains the algebraic tensor product $C_c(X) \odot B$ as a dense subspace. 
The identification $L^2(X , \HH) \cong L^2(X) \otimes \HH$ will be used frequently; 
here, and in the sequel, we denote by $\HHH \otimes \HH$ Hilbertian tensor product of 
Hilbert spaces $\HHH$ and $\HH$. 
We refer to Williams~\cite[Appendix B.I.4]{Wi07} for further details. 

Let $\Bd[\HH,\HHH]$ be the space of all bounded linear operators from $\HH$ into $\HHH$; we write as usual
$\cl B(\HH) = \Bd[\HH,\HH]$. 
For a weak$^*$-closed  subspace $M \subseteq \Bd[\HH,\HHH]$ we let $L^\infty(X,M)$ denote the space of 
(equivalence classes of) bounded functions $f : X \to M$ such that, 
for each $x \in X$ and $\xi \in L^2(X,\HH)$, $\eta\in L^2(X, \HHH)$, 
the functions $x \mapsto f(x)(\xi(x))$ and $x \mapsto f(x)^*(\eta(x))$ are weakly measurable as functions from 
$X$ to $\HH$ and from $X$ to $\HHH$, respectively. 
We equip $L^\infty(X,M)$ with the norm 
$\| f \| \defeq \esssup_{x \in X}\| f(x) \|$ and identify  each $f\in L^\infty(X,M)$ with the operator $D_f$ from $L^2(X,\HH)$ to $L^2(X,\HHH)$ given by $(D_f\xi)(x)=f(x)\xi(x)$. 
See Takesaki~\cite[Section IV.7]{Tak80} for details. 
We write $L^\infty(X,\HH)$ for the space of (equivalence classes of) bounded weakly measurable 
$\HH$-valued functions on $X$. 

Since we have a standing second-countability assumption for locally compact groups (except when we specify a discrete group) our groups are metrisable as topological spaces, and are hence standard measure spaces when equipped with left Haar measure.

\subsubsection{Operator spaces} 

Consider (concrete) operator spaces $V \subseteq \Bd$ and $W \subseteq \Bd[\HHH]$.
The norm-closed spatial tensor product of $V$ and $W$ will be written $V \otimes W$, while if $V$ and $W$ are weak*-closed, their weak*-spatial tensor product will be denoted $V \btens W$.
The operator space projective tensor product $V \projtens W$ satisfies the canonical completely isometric identifications $(V \projtens W)^* = \cbm (V, W^*) = \cbm (W, V^*)$ \cite[Corollary 7.1.5]{EfRu00}; 
if $M$ and $N$ are von~Neumann algebras, $V = M_*$ and $W = N_*$, 
then $(V \projtens W)^* = M \btens N$, up to a complete isometry \cite[Theorem 7.2.4]{EfRu00}.
For $u \in M_n(V \odot W)$ let $\| u \|_h= \inf \{ \| a \| \| b \|\}$, where the infimum is taken over all integers $p$, 
and all matrices $a\in M_{n,p}(V)$ and $b \in M_{p,n}(W)$, such that $u_{i,j} = \sum_k a_{i,k} \otimes b_{k,j}$; 
the Haagerup tensor product $V \otimes^h W$ is the completion  of the operator space
$V \odot W$ in $\| \cdot \|_h$; see \cite[Chapter 9]{EfRu00} for further details. 

For an index set $I$, we will write $C_I^\omega(V)$ for the operator space of families $(x_i)_{i \in I} \subseteq V$ such that the sums $\sum_{i \in J} x_i^* x_i$ are uniformly bounded over 
all finite sets $J \subseteq I$; equivalently, $C_I^\omega(M) = \ell^2(I)_c \btens M$,
where $\ell^2(I)_c$ denotes $\ell^2(I)$, equipped with the column operator space structure. 
Similarly, $R_I^\omega(V)$ denotes the operator space of families 
$(x_i)_{i \in I} \subseteq V$ such that the sums $\sum_{i \in J} x_i x_i^*$ are uniformly bounded 
over all finite sets $J \subseteq I$; equivalently, $R_I^\omega(M) = \ell^2(I)_r \btens M$, where $\ell^2(I)_r$ 
denotes $\ell^2(I)$, equipped with the row operator space structure. 
Further details on the row and column spaces can be found in \cite{EfRu00} and \cite{Pisier-book}.
If $V$ and $W$ are dual operator spaces then their weak* Haagerup tensor product will be written $V \otimes^{w^*h} W$; a typical element $u \in V \otimes^{w^*h} W$ is $u = \sum_{i \in I} f_i \otimes g_i$, where $I$ is some cardinal, $f = (f_i)_{i \in I} \in R^\omega_I(V)$ and $g = (g_i)_{i \in I} \in C^\omega_I(W)$; see \cite{BS92} for further details.

\subsubsection{The trace and Hilbert--Schmidt classes}

Let $\HH$ and $\HHH$ denote Hilbert spaces. 
We write $\Cpt[\HH,\HHH]$ (resp.\ $\Sone(\HH,\HHH)$) 
for the compact (resp.\ trace class) operators from 
$\HH$ to $\HHH$ and simplify $\Cpt[\HH] := \Cpt[\HH,\HH]$, {\it etc.}
The space $\Sone(\HH,\HHH)$ is equipped with the norm $\| T \|_1 := \mathrm{tr}(|T|)$. 
Recall that, \textit{via} trace duality, we have isometric identifications
\[
	\Sone(\HH,\HHH) \cong \Cpt[\HHH , \HH]^* \quad \text{ and } \quad \Bd[\HHH , \HH] \cong \Sone(\HH,\HHH)^*.
\] 
The space of Hilbert--Schmidt operators $T :\HH \to \HHH$, with the norm $\| T \|_2 :=( \mathrm{tr}(T^*T))^{1/2}$, 
will be denoted $\Stwo(\HH,\HHH)$.
These spaces will often appear with $\HH = L^2(X,\mu)$ and $\HHH = L^2(Y,\nu)$, in which case we will write $\Sone(X , Y)$, $\Stwo(X)$, \textit{etc}.

\subsubsection{Crossed products}\label{sss_cp}

Let $A$ be a $C^*$-algebra, viewed as a subalgebra of $B(\HH_A)$, where $\HH_A$ denotes the Hilbert space of the universal representation of $A$. 
Let $G$ be a locally compact group with modular function $\Delta$, equipped with left Haar measure $m_G$,
and $\alpha: G \to \Aut(A)$ be a group homomorphism which is continuous in the point-norm topology, 
\textit{i.e.}\ for all $a\in A$ the map $s\mapsto \alpha_s(a)$ is continuous from $G$ to $A$; 
we say $(A, G,\alpha)$ is a $C^*$-dynamical system. 
The space $L^1(G,A)$ is a Banach $*$-algebra when equipped with the product $\times$ given by
\[
	(f\times g)(t) \defeq \int_G f(s)\alpha_s \big( g(s^{-1}t) \big) ds,\quad f,g\in L^1(G, A),\ t\in G,
\]
the involution $\ast$ defined by
\[
	f^*(s) \defeq \Delta(s)^{-1} \alpha_s \big( f(s^{-1})^* \big), \quad f\in L^1(G, A), s\in G,
\]
and the $L^1$-norm $\|f\|_1 \defeq \int_G \|f(s)\|ds$. 
These definitions also give a $\ast$-algebra structure on $C_c(G, A)$, which is a dense $\ast$-subalgebra of $L^1(G, A)$. 
Given a faithful representation $\theta : A \to \Bd[\HH_\theta]$, 
we define new representations of $A$ and $G$ on $L^2(G,\HH_\theta)$ as follows:
\begin{gather*}
	\pi^\theta : A \to \Bd[L^2(G, \HH_\theta)]; \ \big( \pi^\theta(a) \xi \big)(t) \defeq \theta \big( \alpha_{t\inv}(a) \big) \big( \xi(t) \big) , \\
	\lambda^\theta: G \to \cl B(L^2(G, \HH_\theta));\ (\lambda^\theta_t \xi)(s):= \xi(t^{-1}s),
\end{gather*}
for all $a\in A$, $s,t\in G$ , $\xi\in L^2(G, \HH_\theta)$. 
Then $\lambda^\theta$ is a (strongly continuous) unitary representation of $G$ and 
\[
	\pi^\theta \big( \alpha_t(a) \big) = \lambda^\theta_t \pi^\theta(a) (\lambda^\theta_t)^*, \quad a\in A,\ t\in G.
\]
The pair $(\pi^\theta, \lambda^\theta)$ is thus 
a covariant representation of $(A, G,\alpha)$ and therefore gives rise to a 
$\ast$-representation $\pi^\theta \rtimes \lambda^\theta : L^1(G, A) \to B(L^2(G, \HH_\theta))$ given by 
\[
	(\pi^\theta \rtimes \lambda^\theta)(f) \defeq \int_G \pi^\theta \big( f(s) \big) \lambda^\theta_s ds , \quad f \in L^1(G, A) .
\]
The reduced crossed product $A\rtimes _{\alpha, r}G$
of $A$ by $G$ is independent of the choice of the faithful representation $\theta$
and is defined as the closure of $(\pi^\theta \rtimes \lambda^\theta)(L^1(G, A))$ in the operator norm of $B(L^2(G, \HH_\theta))$; 
if we want to emphasise the representation $\theta$ of $A$ was used, we will write $A\rtimes _{\alpha, \theta}G$.
In Section~\ref{sec:convmults} we will use the weak* closure $A\rtimes_{\alpha, r}^{w^{\ast}} G$ 
of $A\rtimes _{\alpha, r}G$.
In what follows we will often simplify our notation by omitting the superscript $\theta$. 
More on reduced crossed products can be found in Pedersen~\cite[Chapter 7]{Pe79}, and Williams~\cite{Wi07}.

\subsection{Multipliers}
\label{ssec:multiplierprelims}

We will use some well-known results on classical Schur and Herz--Schur multipliers, as well as results 
from \cite{mtt}. We recall some definitions and results required later.

\subsubsection{Schur multipliers}

Let $(X,\mu)$ and $(Y,\nu)$ be standard measure spaces.
We say $E \subseteq X \times Y$ is \emph{marginally null} if there exist null sets $M \subseteq X$ and $N \subseteq Y$ such that $E \subseteq (M \times Y) \cup (X \times N)$. 
Two measurable sets $E,F \subseteq X \times Y$ are called \emph{marginally equivalent} if their symmetric difference is marginally null; 
we say that two functions $\varphi , \psi : X \times Y \to \CC$ are \emph{marginally equivalent} if they are equal up to a marginally null set. 
A measurable set $E \subseteq X \times Y$ is called \emph{$\omega$-open} 
if it is marginally equivalent to a set of the form $\cup_{k \in \NN} I_k\times J_k$, where $I_k \subseteq X$ and $J_k \subseteq Y$ are measurable, $k\in \mathbb{N}$. 
The collection of $\omega$-open subsets of $X \times Y$ is a \emph{pseudo-topology} on $X \times Y$ --- it is closed under finite intersections and countable unions; see \cite[Section 3]{EKS}.
A function $h : X\times Y\to \mathbb{C}$ is called {\it $\omega$-continuous} \cite{EKS} if 
$h^{-1}(U)$ is $\omega$-open for every open set $U\subseteq \mathbb{C}$.

Let $\HH$ be a separable Hilbert space and $A\subseteq \Bd$ be a separable $C^*$-algebra. 
With any $k\in L^2(Y\times X, A)$, one can associate an element $T_k\in \Bd[L^2(X, \HH), L^2(Y, \HH)]$ with $\|T_k\|\leq \|k\|_2$, 
by letting
\[
	(T_k \xi)(y) \defeq \int_X k(y,x) (\xi(x)) dx , \quad \xi \in L^2(X, \HH),\ y \in Y.
\]
The linear space of all such operators is denoted by $\Stwo(X,Y ; A)$ and is norm dense in $\Cpt[L^2(X), L^2(Y)]\otimes A$;
we equip it with the operator space structure arising from this inclusion. 
Note that if $A = \mathbb{C}$ then 
the map $k\to T_k$ is an isometric identification of $L^2(Y \times X)$ and $\Stwo(X,Y)$.

If $B$ is a(nother) $C^*$-algebra we write $\cbm(A,B)$ for the space of completely bounded maps from $A$ to $B$ 
and set $\cbm(A) = \cbm(A,A)$. 
We say that $\varphi:X\times Y\to \cbm(A,B)$ is pointwise-measurable if $(x,y)\mapsto\varphi(x,y)(a)\in B$ is weakly measurable for each $a\in A$. 
If $\varphi: X\times Y \to \cbm(A)$ is a bounded, pointwise-measurable function, 
we define $\varphi \cdot k \in L^2(Y\times X, A)$ by
\[
	(\varphi \cdot k)(y,x) \defeq \varphi(x,y) \big( k(y,x) \big), \quad (y, x) \in Y \times X . 
\]
Let $S_\varphi$ denote the bounded linear map on $\Stwo(X,Y ; A)$ given by
\[
	S_\varphi (T_k) \defeq T_{\varphi \cdot k}, \quad k\in L^2(Y\times X, A) .
\]
 
\begin{definition}\label{de:SchurAmult}
A bounded, pointwise-measurable function $\varphi: X\times Y \to \cbm(A)$ is 
called a \emph{Schur $A$-multiplier} if $S_\varphi$ is a completely bounded map on $\Stwo(X,Y; A)$. 
We denote the space of such functions by $\mathfrak{S}(X,Y;A)$ and endow it with the norm $\|\varphi\|_{\Sch(X,Y;A)} := \|S_{\varphi}\|_\cb$ (we write $\|\varphi\|_{\Sch}$ when $X,Y$ and $A$ are clear from context). 
\end{definition}

\noindent
This definition does not depend on the faithful $*$-representation of $A$ on a separable Hilbert space  \cite[Proposition 2.3]{mtt}.

\begin{theorem}\cite[Theorem 2.6]{mtt}\label{th:Schurmultsmtt}
Let $A\subseteq \Bd$ be a separable $C^*$-algebra and $\varphi : X\times Y \to \cbm(A)$ a bounded, pointwise measurable  function. 
The following are equivalent:
\begin{enumerate}[i.] 
	\item $\varphi$ is a Schur $A$-multiplier;
 	\item there exist a separable Hilbert space $\HH_\rho$, a non-degenerate $*$-representa- tion 
	$\rho: A \to \cl B(\HH_\rho)$, and $\mathcal{V}\in L^\infty(X, \Bd[\HH,\HH_\rho])$, $\mathcal{W}\in L^\infty(Y, \Bd[\HH,\HH_\rho])$ such that
 	\[
 		\varphi(x,y)(a) = \mathcal{W}(y)^* \rho(a) \mathcal{V}(x), \quad a\in A
 	\]
 	for almost all $(x,y)\in X\times Y$.
\end{enumerate}
Moreover, if these conditions hold then we may choose $\mathcal{V}$ and $\mathcal{W}$ so that
\[
	\| \varphi \|_\Sch = \esssup_{x \in X} \| \mathcal{V}(x) \| \esssup_{y \in Y} \| \mathcal{W}(y) \| .
\]
\end{theorem}

Note that the definitions and theorems make sense in the case $X$, $Y$ are discrete spaces with counting measures, in which case we do not need to assume separability. 

When discussing Schur $A$-multipliers we shall always assume without mentioning that $A$ is separable unless $X$ and $Y$ are discrete spaces with counting measures in which case $A$ can be arbitrary.

In the case where $A = \mathbb{C}$, Schur $A$-multipliers reduce to classical (measurable) Schur multipliers 
\cite{Peller}.
The elements $\sum_{i=1}^{\infty} f_i \otimes g_i$ 
of the projective tensor product $\Sone(Y,X) = L^2(X,\mu) \projtens L^2(Y,\mu)$ 
(where we assume $\sum_{i=1}^{\infty} \|f_i\|^2 < \infty$ and $\sum_{i=1}^{\infty} \|g_i\|^2 < \infty$)
can be identified with functions $\sum_{i=1}^{\infty} f_i(x) g_i(y)$ on $X\times Y$, well-defined up to a 
marginally null set \cite{Arv74}; under this identification, 
Schur multipliers coincide with the multipliers of $\Sone(Y,X)$.

Given $a\in L^{\infty}(X, \mu)$, let $M_a$ be the operator on $L^2(X, \mu)$ defined by 
$$(M_a \xi)(x) \defeq a(x) \xi(x), \ \ \ x\in X.$$
Let $\mathcal{D}_X= \{ M_a : a\in L^\infty(X, \mu) \}$ and define $\mathcal{D}_Y$ analogously.
By a well-known result of Haagerup~\cite{Ha80} (see also\cite{BS92}), 
there is a completely isometric weak*-homeomorphism between the algebra of weak*-continuous, completely bounded $\mathcal{D}_Y,\mathcal{D}_X$-bimodule maps on $\Bd[L^2(X),L^2(Y)]$ and the weak* Haagerup tensor product 
$\mathcal{D}_Y\otimes^{w^*h}\mathcal{D}_X$ \cite{BS92}; 
this homeomorphism sends
$\sum_{k=1}^{\infty} b_k\otimes a_k\in \mathcal{D}_Y \otimes^{w^*h} \mathcal{D}_X$ to the map 
\[
T \mapsto \sum_{k=1}^{\infty} b_k T a_k
\]
on $\Bd[L^2(X),L^2(Y)]$.
Note that $\mathcal{D}_Y\otimes^{w^*h}\mathcal{D}_X$ can be viewed as a space of 
(equivalence classes of) functions, and each of these functions belongs to $\mathfrak{S}(X,Y)$. 
Theorem \ref{th:Schurmultsmtt} can be specialised as follows in the scalar-valued case.
  
\begin{theorem}
Let $\varphi \in L^{\infty}(X\times Y)$. The following are equivalent:
\begin{enumerate}[i.]
  		\item $\varphi\in \mathfrak{S}(X, Y)$ and $\|\varphi\|_{\mathfrak{S}} \leq C$;
  		\item there exists sequences $(a_k)_{k=1}^{\infty}\subseteq L^{\infty}(X, \mu)$ and $(b_k)_{k=1}^{\infty}\subseteq L^{\infty}(Y, \nu)$ with
  		\[
  			\esssup_{x\in X}\sum_{k=1}^{\infty}|a_k(x)|^2\leq C \quad \text{and} \quad \esssup_{y\in Y}\sum_{k=1}^{\infty}|b_k(y)|^2\leq C
  		\]
  		such that 
  		\[
  			\varphi(x,y) = \sum_{k=1}^{\infty} a_k(x)b_k(y) \quad \text{for almost all $(x,y) \in X \times Y$};
  		\]
  		\item there exist a separable Hilbert space $\HH$ and weakly measurable functions $v: X \to \HH$, $w: Y \to \HH$, such that
  		\[ 
  			\esssup_{x\in X} \|v(x)\| \leq \sqrt{C},\quad \esssup_{y\in Y} \|w(y)\| \leq \sqrt{C}
  		\]
  		and 
  		\[
  			\varphi(x,y) = \ip{v(x)}{w(y)}, \quad \text{for almost all $(x,y) \in X \times Y$}; 
  		\]
  		\item $\| T_{\varphi \cdot k} \| \leq C \| T_k \|$ for all $k\in L^2(Y\times X)$.
\end{enumerate}	 
\end{theorem}     

We remark that if $X$ and $Y$ are discrete spaces with counting measures 
the theorem holds true with possibly uncountable families $(a_k)$ and $(b_k)$.

\subsubsection{Herz--Schur multipliers}

Let $G$ be a locally compact second countable group, $\vN(G)$ (resp. $C^*_r(G)$) be its von Neumann algebra
(resp. reduced $C^*$-algebra) and $A(G)$ be the Fourier algebra of $G$ \cite{eymard}. 
Let $A$ be a separable $C^*$-algebra. 
A bounded function $F: G \to \cbm(A)$ will be called pointwise-measurable if, for every $a\in A$, the map $s\mapsto F(s)(a)$ is a weakly measurable function from $G$ into $A$. 
 Suppose that the function $F : G \to \cbm(A)$ is bounded and pointwise-measurable, and define  
\[
	(F \cdot f)(s) \defeq F(s) \big( f(s) \big) , \quad f \in L^1(G,A),\ s \in G.
\]
Since $F$ is pointwise-measurable, $F \cdot f$ is weakly measurable, and $\| F \cdot f \|_1 \leq \sup_{s\in G} \|F(s)\| \| f \|_1$ ($f \in L^1(G,A)$); hence $F \cdot f\in L^1(G,A)$ for every $f\in L^1(G,A)$. 

\begin{definition}\label{de:HerzSchurmultofsystem}
A bounded, pointwise measurable function $F: G \to \cbm(A)$ will be called a \emph{Herz--Schur $(A, G,\alpha)$-multiplier} 
if the map $S_F$ on $(\pi\rtimes \lambda)(L^1(G, A))$, given by 
\[
	S_F \big( (\pi\rtimes \lambda)(f) \big) \defeq (\pi\rtimes \lambda)(F \cdot f),
\] 
is completely bounded. 
 \end{definition}

If $F$ is a Herz--Schur $(A, G,\alpha)$-multiplier, we continue to denote by $S_F$ the corresponding extension to a 
completely bounded map on $A\rtimes_{\alpha, r} G$.

Definition~\ref{de:HerzSchurmultofsystem} is independent of the faithful representation of $A$ \cite[Remark 3.2(ii)]{mtt}. 
We note that the set of all Herz--Schur $(A,G,\alpha)$-multipliers is an algebra with respect to the
pointwise operations; we denote it by $\mathfrak{S}(A,G,\alpha)$ and endow it with the norm $\|F\|_\HS \defeq \|S_F\|_\cb$.

The definition makes sense when $G$ is an arbitrary discrete group. In this case we can drop the separability assumption on $A$.  

In what follows we shall always consider $C^*$-dynamical systems $(A,G,\alpha)$ where either $G$ is second countable and $A$ is separable or $G$ is discrete in which case $A$ can be arbitrary. 

Given a function $F: G\to \cbm(A)$, define 
$\mathcal{N}(F): G\times G \to \cbm(A)$ by letting
\[
\mathcal{N}(F)(s,t)(a) \defeq \alpha_{t\inv} \big( F(ts^{-1})\big( \alpha_t(a) \big) \big), \quad s,t\in G,\ a\in A .
\]
Observe that if $F$ is pointwise measurable then so is $\mathcal{N}(F)$. 
The following result \cite[Theorem 3.5]{mtt} relates Schur $A$-multipliers and Herz--Schur $(A,G, \alpha)$-multipliers, generalising 
a classical transference result of Bo\.{z}ejko--Fendler \cite{BoFe84}.
  
\begin{theorem}\label{th:transferencemtt}
Let $(A,G, \alpha)$ be a $C^*$-dynamical system and $F: G\to \cbm(A)$ a bounded, pointwise-measurable function. 
The following are equivalent:
\begin{enumerate}[i.]
	\item $F$ is a Herz-Schur $(A,G,\alpha)$-multiplier;
	\item $\mathcal{N}(F)$ is a Schur $A$-multiplier.
\end{enumerate}	
Moreover, if the above conditions hold then $\|F\|_{\rm{HS}}=\|\mathcal{N}(F)\|_{\mathfrak{S}}$.
\end{theorem}

The Schur $A$-multipliers $\nph$ of the form $\nph = \mathcal{N}(F)$ will be called \emph{$\alpha$-invariant}. 
We note that a different definition was given in \cite{mtt} (see \cite[Definition 3.14]{mtt}), but by 
\cite[Theorem 3.18]{mtt}, it agrees with the one adopted here.

In the case where $A = \mathbb{C}$ and the action is trivial, Herz--Schur $(A,G,\alpha)$-multipliers coincide with 
the classical \emph{Herz--Schur multipliers} of $G$ \cite{CaHa85}, that is, with the functions
$u : G\to \mathbb{C}$ such that $u A(G)\subseteq A(G)$ and the map 
\[
	m_u: A(G) \to A(G) ; \ m_u(v) \defeq uv , \quad v \in A(G) ,
\]
is completely bounded. 
Here we equip $A(G)$ with the operator space structure, arising from the identification $A(G)^* = \vN(G)$ \cite{eymard}.
The space of classical Herz--Schur multipliers of $G$ will be denoted by $\Mcb A(G)$.  
We note that if $u\in \Mcb A(G)$ then
the restriction $S_u := m_u^*|_{C^*_r(G)}$ is a completely bounded map satisfying \cite{CaHa85}
\[
	S_u : C^*_r(G) \to C^*_r(G) ; \ S_u \big( \lambda(f) \big) = \lambda( uf) , \quad f \in L^1(G) .
\]

\subsection{Preliminary results}
\label{ssec:prelimlemmas}

In this subsection, 
we give several technical results on Schur and Herz--Schur multipliers that will be needed in the sequel. 
The equivalence between (i) and (iii) in the next proposition was given, in the scalar-valued case, in 
\cite[Theorem 7]{KP05}.

\begin{proposition}\label{pr:zeroalmosteverywhere}
Let $\HH$ be a separable Hilbert space, $A \subseteq \cl B(\HH)$ a separable $C^*$-algebra and $\varphi : X \times Y \to \cbm(A)$ a bounded, pointwise-measurable function.
The following are equivalent:
\begin{enumerate}[i.]
	\item $\varphi(x,y) = 0$ for almost all $(x,y) \in X \times Y$;
	\item $S_{\varphi} = 0$.
\end{enumerate}
If $\varphi$ is a Schur $A$-multiplier of the form 
$\varphi(x,y)(a) = \mathcal{W}(y)^* \rho(a) \mathcal{V}(x)$, $a\in A$, 
as in Theorem \ref{th:Schurmultsmtt}, 
then these conditions are equivalent to:
\begin{enumerate}[i.] \setcounter{enumi}{2}
	\item $\varphi(x,y) = 0$ for marginally almost all $(x,y) \in X \times Y$.
\end{enumerate}

\end{proposition}
\begin{proof}
(i)$\implies$(ii) Let $T_k \in \Stwo(X,Y ; A)$.
If $\varphi(x,y) = 0$ for almost all $(x,y) \in X \times Y$ then 
$\varphi\cdot k = 0$ almost everywhere, for every $k\in L^2(Y\times X, A)$, and hence
$S_\varphi(T_k) = T_{\varphi\cdot k} = 0$ for every $k\in L^2(Y\times X, A)$. 

(ii)$\implies$(i) Suppose $S_\varphi= 0$ and let 
$k\in L^2(Y\times X, A)$. 
We have $S_\varphi(T_k) = T_{\varphi \cdot k} =0$, so we conclude that $\varphi \cdot k = 0$ 
almost everywhere by \cite[Lemma 2.1]{mtt}. 
We claim that $\varphi(x,y) =0$ for almost all $(x,y) \in X \times Y$.
Indeed, let $\{ e_i \}_{i \in \mathbb{N}}$ be a dense subset of $\HH$, $\xi\in L^2(X)$ and $\eta\in L^2(Y)$; then 
\begin{eqnarray}\label{eq_3l}
	& & \ip{S_\varphi(T_k) (\xi \otimes e_i)}{\eta \otimes e_j} 
	=  
	\int_Y \ip{S_\varphi(T_k) (\xi \otimes e_i)(y)}{(\eta \otimes e_j)(y)} dy \\
	& = &
	\int_Y \ip{\int_X (\varphi \cdot k)(y,x)(\xi \otimes e_i)(x) dx}{(\eta \otimes e_j)(y)} dy \nonumber\\
        & = & \int_{Y}\int_{X} \ip{\varphi(x,y) \big( k(y,x) \big)e_i}{e_j} \xi(x)\overline{\eta(y)} dx\, dy.\nonumber
\end{eqnarray}
Fix $a \in A$, choose $w \in L^2(Y \times X)$, and let $k(y,x)= w(y,x)a$.
Then (\ref{eq_3l}) implies
$$\int_Y\int_X \ip{\varphi(x,y)(a)e_i}{e_j} w(y,x) \xi(x) \overline{\eta(y)} dx\, dy = 0.$$
Since $\varphi(x,y)(a)$ is a bounded operator, 
we conclude that $\langle \varphi(x,y)(a) e_i , e_j \rangle = 0$ almost everywhere for all $i,j \in \mathbb{N}$.
Hence $\varphi(x,y) = 0$ almost everywhere by the separability of $A$ and the continuity of $\varphi(x,y)$.

Now suppose that $\varphi$ is a Schur $A$-multiplier.

(iii)$\implies$(i) is trivial. 

(i)$\implies$(iii) Assume that the set
$$R \defeq \{ (x,y) \in X \times Y : \varphi(x,y) \neq 0 \}$$
is null. 
Let $A_0$ and $\HH_0$ be countable dense subsets of $A$ and $\HH$ respectively; then
\[
\begin{split}
    R^c = \{ (x,y) : \varphi(x,y) = 0 \} &= \bigcap_{a \in A_0 , \ \xi,\eta \in \HH_0} \{ (x,y) : \ip{\varphi(x,y)(a) \xi}{\eta} = 0 \} \\
        &= \bigcap_{a \in A_0 , \ \xi,\eta \in \HH_0} \{ (x,y) : \ip{\rho(a) \mathcal{V}(x) \xi}{\mathcal{W}(y) \eta} = 0 \} .
\end{split}
\]
It is easily seen that a function of the form 
$(x,y) \mapsto \ip{\alpha(x)}{\beta(y)}$, where $\alpha \in L^\infty(X , \HH_\rho)$ and $\beta \in L^\infty(Y , \HH_\rho)$, 
is $\omega$-continuous; thus, the set $\{ (x,y) : \ip{ \alpha(x)}{ \beta(y)} \neq 0 \}$ is $\omega$-open.
It follows that the set
\[
    \bigcup_{a \in A_0 , \ \xi,\eta \in \HH_0} \{ (x,y) : \ip{\rho(a) \mathcal{V}(x) \xi}{\mathcal{W}(y) \eta} \neq 0 \}
\]
is $\omega$-open.
Hence there are families $A_n \subseteq X, \ B_n \subseteq Y$ of measurable sets such that $R$ is marginally equivalent to $\cup_{n =1}^\infty A_n \times B_n$.
Since $(\mu \times \nu)(R) = 0$ we have $\mu(A_n) \nu(B_n) = 0$ for each $n$.
Let
\[
    N_1 \defeq \bigcup_{\nu(B_n) \neq 0} A_n \quad \text{ and } \quad N_2 \defeq \bigcup_{\mu(A_n) \neq 0} B_n.
\]
We have that $\mu(N_1) = 0$, $\nu(N_2) = 0$ and $R$ that 
is marginally equivalent to a subset of $N_1 \times Y \cup X \times N_2$; thus, $R$ is marginally null.
\end{proof}



The next lemma contains a completely isometric version of the main transference result of \cite[Section 3]{mtt}.

\begin{lemma}\label{le:Ncompleteisomhom}
Let $(A,G,\alpha)$ be a $C^*$-dynamical system.
The map $\cl N$ is a completely isometric algebra homomorphism from the space of Herz--Schur $(A,G,\alpha)$-multipliers to the space of Schur $A$-multipliers on $G \times G$.
\end{lemma}

\begin{proof}
Fix $n \in \mathbb{N}$ and Herz--Schur $(A,G,\alpha)$-multipliers $F_{i,j}$, $1 \leq i,j \leq n$.
Since $(S_{F_{i,j}})_{i,j}$ is an element of $\cbm (A \rtimes_{\alpha ,r} G , M_n(A \rtimes_{\alpha ,r} G))$ 
there exist a representation $\rho: A \rtimes_{\alpha ,r} G \to \cl B(\cl H_{\rho})$ 
and operators $V, W : L^2(G,\cl H)\to \cl H_{\rho}$ 
such that $(S_{F_{i,j}}) _{i,j} = W^* \rho( \cdot ) V$ and $\| V \| \| W \| = \| (S_{F_{i,j}})_{i,j} \|_\cb$.
Take $a \in A$ and $r \in G$. Arguing as in the proof of \cite[Theorem 3.8]{mtt} we obtain representations $\rho_A$ and $\rho_G$, of $A$ and $G$ respectively, such that
\[
    \big( \pi(F_{i,j}(t)(a)) \lambda_r \big)_{i,j} = \big( S_{F_{i,j}}(\pi(a) \lambda_r) \big)_{i,j} = W^* \rho_A(a) \rho_G(r) V .
\]
Define
\[
    \mathcal{V}(s) := \rho_G(s\inv) V \lambda_s \quad \text{ and } \quad \mathcal{W}(t) := \rho_G(t\inv) W \lambda_t ,
\]
so that $\sup_{s \in G} \| \mathcal{V}(s) \| \sup_{t \in G} \| \mathcal{W}(t) \| = \| V \| \| W \| = \|(S_{F_{i,j}})_{i,j}\|_\cb$.
Calculations as in the proof of \cite[Theorem 3.8]{mtt} show that
\[
    (\mathcal{N}(F_{i,j})(s,t)(a))_{i,j} = \mathcal{W}(t)^* \rho_A(a) \mathcal{V}(s) ,
\]
almost everywhere, so
\[
    \| (S_{\mathcal{N}(F_{i,j})}) _{i,j}\|_\cb \leq \sup_{s \in G} \| \mathcal{V}(s) \| \sup_{t \in G} \| \mathcal{W}(t) \| = \| V \| \| W \|  = \| (S_{F_{i,j}}) _{i,j}\|_\cb .
\]
In the converse direction, 
note that 
$(S_{F_{i,j}})_{i,j}$ is the restriction of $(S_{\mathcal{N}(F_{i,j})})_{i,j}$ to $M_n(A \rtimes_{\alpha,r} G)$, so $\| (S_{F_{i,j}})_{i,j} \|_\cb \leq \| (S_{\mathcal{N}(F_{i,j})})_{i,j} \|_\cb$.
Thus $F \mapsto \mathcal{N}(F)$ is a complete isometry.
The homomorphism claim is trivial.
\end{proof}

\section{Central multipliers}\label{sec:centmults}

Let $(X,\mu)$ and $(Y,\nu)$ be standard measure spaces. 
We  denote for brevity by $\cl B$ (resp.\ $\cl K$) the space $\Bd[L^2(X,\mu),L^2(Y,\nu)]$ (resp.\ $\Cpt[L^2(X,\mu),L^2(Y,\nu)]$).
Throughout this section $A$ denotes a separable $C^*$-algebra, acting non-degenerately on a separable Hilbert space $\HH$.
The multiplier algebra of $A$ will be written $\Mult(A)$ and identified with the idealiser of $A$ in $\Bd$:
\[
	\Mult(A) = \{ c \in \Bd : ca ,ac \in A  \text{ for all $a \in A$} \} .
\]
As usual, we denote by $Z(B)$ the centre of the $C^*$-algebra $B$.

The following is immediate, and will be used several times in the sequel. 

\begin{remark}\label{r_suba}
{\rm
Let $B\subseteq A$ be a $C^*$-subalgebra, and $\varphi : X \times Y\to \cbm(A)$ be a Schur $A$-multiplier.
Suppose that $\varphi(x,y)$ leaves $B$ invariant for almost all $(x,y)$, and let $\varphi_B : X \times Y\to \cbm(B)$ be the map given by $\varphi_B(x,y)(b) := \varphi(x,y)(b)$ ($b\in B$, $(x,y)\in X\times Y)$.
Then  $\varphi_B$ is a Schur $B$-multiplier and $\|\varphi_B\|_\Sch \leq \| \varphi \|_\Sch$.
}
\end{remark}

\subsection{Central Schur multipliers}
\label{ssec:centralSchur}

\begin{definition}\label{d_ce}
A Schur $A$-multiplier $\nph\in \mathfrak{S}(X,Y;A)$ will be called \emph{central} if there exists a family $(a_{x,y})_{(x,y)\in X\times Y} \subseteq Z(\Mult(A))$ such that 
\begin{equation}\label{eq_axy}
    \nph(x,y)(a) = a_{x,y}a, \quad a\in A.
\end{equation}
\end{definition}

\begin{remark}\label{r_ce}
{\rm
Let $\nph\in \mathfrak{S}(X,Y;A)$  be a central Schur $A$-multiplier.
\begin{enumerate}[i.]
    \item The family $(a_{x,y})_{(x,y)\in X\times  Y}$ associated to $\nph$ in Definition \ref{d_ce} is unique up to a set of zero product measure. 
    \item If $(a_{x,y})_{(x,y)\in X\times  Y}$ is associated to $\nph$ as in Definition \ref{d_ce} then the map $X\times Y \to Z(\Mult(A))$, $(x,y) \mapsto a_{x,y}$, is weakly measurable. 
\end{enumerate}
}
\end{remark}

Let $A$ be a commutative $C^*$-algebra, and assume that $A = C_0(Z)$, 
where $Z$ is a locally compact Hausdorff space. 
The standing separability assumption implies 
that $Z$ is second-countable, and hence metrisable. 
Since $C_0(Z)$ is separable it has a faithful state, so the associated Radon measure $m$ on $Z$ has full support. 

Let $C_0(Z,B)$ be the space of all continuous functions from $Z$ into a normed space $B$  
vanishing at infinity. 
We write $\cl K = \cl K(L^2(X), L^2(Y))$ and note that, 
up to a canonical $*$-isomorphism, 
\begin{equation}\label{eq_KC}
\cl K \otimes C_0(Z) = C_0(Z,\cl K). 
\end{equation}
The algebraic tensor product $L^2(Y\times X)\odot C_0(Z)$ can thus be viewed as 
a (dense) subspace of any of the spaces $\cl K \otimes C_0(Z)$ and $C_0(Z,\cl K)$. 

Let $\nph\in \mathfrak{S} (X,Y;$ $C_0(Z))$ be a central Schur $C_0(Z)$-multiplier, 
associated with a family $(a_{x,y})_{(x,y)\in X\times  Y} \subseteq C_b(Z)$ as in Definition~\ref{d_ce}; 
we view $\varphi$ as a scalar-valued function 
on $X\times Y \times Z$ by letting
\[
\nph(x,y,z) =  a_{x,y}(z), \quad x \in X,\ y\in  Y,\ z\in Z.
\]
By definition, $\varphi$ is a bounded, measurable function on $X \times Y \times Z$ which is continuous in the $Z$-variable. 
On the other hand, suppose $\nph : X\times Y \times Z\to \CC$ is a bounded measurable function, continuous in the $Z$-variable. 
Then $(x,y)\mapsto \nph(x,y,\cdot)a(\cdot)\in C_0(Z)$ is weakly measurable for each $a\in C_0(Z)$. 
Indeed, the function $(x,y) \mapsto \delta_z(\nph(x,y)(a))=\nph(x,y,z)a(z)$ is measurable for each $z\in Z$
(here $\delta_z$ stands for the point mass measure at $z\in Z$).
As any $m\in M(Z) = C_0(Z)^*$ is the weak* limit of linear combinations of 
point mass measures, we conclude that the function 
$(x,y)\mapsto m(\nph(x,y)(a))$ is measurable for all $m\in M(Z)$.
We thus identify the central Schur $C_0(Z)$-multipliers with
bounded measurable functions $\nph : X \times Y \times Z \to \CC$, continuous in the $Z$-variable.
For each $z\in  Z$, let $\nph_z : X\times  Y\to \CC$  be  given by $\nph_z(x,y) = \nph(x,y,z)$; 
clearly, $\nph_z$ is a measurable function for  each $z\in Z$.

We recall some terminology from \cite{CleMS} that will be used in the sequel. 
Let $\varphi \in L^\infty(X \times Y \times Z)$ and associate with it a bounded bilinear map
\[
	\Lambda_\varphi : \Stwo(Y , Z) \times \Stwo(X,Y) \to \Stwo(X,Z) ; \ \Lambda_\varphi(T_h , T_k) := T_{\varphi (h * k)} , 
\]
where $k \in L^2(Y \times X),\ h \in L^2(Z \times Y)$ and 
\[
	\varphi (h * k) (z,x) := \int_Y \varphi(x,y,z) h(z,y) k(y,x) dy , \quad (x,z) \in X \times Z .
\]
By \cite[Corollary 10]{CleMS} 
the norm $\| \Lambda_\varphi \|$ of $\Lambda_\varphi$ as a bilinear map, where 
the spaces $\Stwo(Y , Z)$ and $\Stwo(X,Y)$ are equipped with their Hilbert-Schmidt norm,
is equal to $\| \varphi \|_\infty$. 
We say that $\varphi$ is an \emph{operator $\Sone$-multiplier} if 
$\Lambda_\varphi$ maps $\Stwo(Y,Z) \times \Stwo(X,Y)$ into $\Sone(X,Z)$. 
The following characterisation of operator $\Sone$-multipliers was obtained in \cite{CleMS}:

\begin{theorem}\label{th_dia0}
Let $\nph : X\times Y \times Z\to \CC$ be a bounded measurable function.  
The following are  equivalent:
\begin{enumerate}[i.]
    \item the function $\nph$ is an operator $\mathcal{S}_1$-multiplier;
    
\item there exist a Hilbert space $\HHH$ and weakly measurable functions $v : X\times Z\to \HHH$, $w : Y\times Z\to \HHH$, satisfying 
    \[
    	\esssup_{(x,z) \in X \times Z} \|v(x,z)\|  < \infty, \ \esssup_{(y,z) \in Y \times Z} \|w(y,z)\|  < \infty ,
    \]
    such that
\begin{equation}\label{eq_vw}
        \nph(x,y,z) = \ip{v(x,z)}{w(y,z)}, \quad \text{ almost all $(x,y,z) \in X \times Y \times Z$.} 
\end{equation}
\end{enumerate}
Moreover, if these conditions hold then 
$$\| \varphi \|_\Sch = 
\esssup_{(x,z) \in X \times Z} \| v(x,z) \| \esssup_{(y,z) \in Y \times Z} \| w(y,z) \|.$$
\end{theorem}

In Theorem \ref{th_dia}, we relate operator $\mathcal{S}_1$-multipliers to central multipliers.
We first include a lemma. If $\cl E$ is an operator space then we identify 
$C_0(Z)\odot \cl E$ with a dense subspace of the minimal tensor product $C_0(Z)\otimes \cl E$
(and equip it with the operator space structure arising from this inclusion),
and its elements --- with continuous functions from $Z$ into $\cl E$. 
If $\cl E$ is in addition an operator system, 
we equip the algebraic tensor product
$C_0(Z)\odot \cl E$ with the operator system structure arising from its inclusion in 
$C_0(Z)\otimes\cl E$.

\begin{lemma}\label{l_diagon}
Let $Z$ be a locally compact Hausdorff space and $\cl E$ be an operator space. 
Let $\Phi_z : \cl E\to \cl E$ be a linear map, $z\in Z$, and 
$\Phi : C_0(Z)\odot \cl E\to C_0(Z)\otimes \cl E$ 
a linear map defined by 
$$\Phi(a\otimes T)(z) = a(z)\Phi_z(T), \ \ z\in Z.$$
The following are equivalent:
\begin{enumerate}[i.]
    \item $\Phi$ is completely bounded;
    \item $\Phi_z$ is completely bounded for every $z\in Z$ and $\sup_{z\in Z}\|\Phi_z\|_{\rm cb} < \infty$.
\end{enumerate}
Moreover, if these conditions are fulfilled then 
$\|\Phi\|_{\rm cb} = \sup_{z\in Z}\|\Phi_z\|_{\rm cb}$.

Assume that $\cl E$ is an operator system. 
The following are equivalent:
\begin{enumerate}[i'.]
    \item $\Phi$ is completely positive;
    \item $\Phi_z$ is completely positive for every $z\in Z$. 
\end{enumerate}
\end{lemma}
\begin{proof}
(i)$\implies$(ii) 
Fix $z\in Z$ and note that, if $a\in C_0(Z)$ has norm one and $a(z) = 1$ then 
$$\Phi_z(T) = (\delta_z\otimes\id)(\Phi(a\otimes T)), \ \ \ T\in \cl E.$$
It follows that $\Phi_z$ is completely bounded and 
\begin{equation}\label{eq_ine}
    \sup_{z\in Z}\|\Phi_z\|_{\rm cb}\leq \|\Phi\|_{\rm cb}. 
\end{equation}

(ii)$\implies$(i) 
We identify $M_n(C_0(Z)\odot \cl E)$ with a subspace of $C_0(Z,M_n(\cl E))$ in the canonical way. 
Let $(h_{i,j})_{i,j}\in M_n(C_0(Z)\odot \cl E)$. The claim is immediate from the fact that
$$\Phi^{(n)}\left((h_{i,j})_{i,j}\right)(z)
= \left(\Phi(h_{i,j})(z)\right)_{i,j}
= \left(\Phi_z(h_{i,j}(z))\right)_{i,j}.$$

It remains to note the reverse inequality in (\ref{eq_ine}); 
it follows by the fact that, if $a\in C_0(Z)$ has norm one and $a(z) = 1$ then
$\|\Phi_z^{(n)}(T)\| \leq \|\Phi^{(n)}(a\otimes T)\|$, for every $T\in M_n(\cl E)$.

Now assume that $\cl E$ is an operator system. 

(i')$\implies$(ii') follows as the implication (i)$\implies$(ii), by choosing 
the function $a$ to be in addition positive. 

(ii')$\implies$(i') follows similarly to the implication (ii)$\implies$(i), by taking into account that 
a matrix $(h_{i,j})_{i,j}$ belongs to the positive cone of $M_n(C_0(Z)\odot \cl E)$ if and only if 
$(h_{i,j}(z))_{i,j}\in M_n^+$ for every $z\in Z$. 
\end{proof}

\begin{theorem}\label{th_dia}
Let $\nph : X\times Y \times Z\to \CC$ be a bounded measurable function, continuous in the $Z$-variable.
The following are  equivalent:
\begin{enumerate}[i.]
    \item $\nph$ is a  central Schur $C_0(Z)$-multiplier; 
    \item the function $\nph_z$ is a Schur multiplier for every $z\in Z$, and the map $D_{\nph} : C_0(Z,\cl K) \to C_0(Z,\cl K)$ given by
    \[
        D_{\nph}(h)(z)  = S_{\nph_z}(h(z)), \quad z\in Z,
    \]
is completely  bounded;
    \item the function $\nph_z$ is a Schur multiplier for every $z\in Z$, and 
       \[
        \sup_{z\in Z} \|\nph_z\|_\Sch < \infty;
    \]
    \item the function $\nph$ is an operator $\mathcal{S}_1$-multiplier. 
\end{enumerate}
If these conditions hold then $\| \varphi \|_\Sch = \sup_{z \in Z} \| \varphi_z \|_\Sch$.
\end{theorem}
\begin{proof}
(i)$\iff$(ii) 
Let $\varphi$ be a central Schur $C_0(Z)$-multiplier. 
We fix a measure $m\in M(Z)$ so that the 
representation of $C_0(Z)$ on $L^2(Z,m)$, given by 
$a \mapsto M_a$, where 
\[
    (M_a \xi)(z) := a(z)\xi(z),  \quad a\in C_0(Z),\ \xi\in L^2(Z, m),\ z\in Z,
\]
is faithful. 
By \cite[Proposition 2.3]{mtt}, 
we may identify $C_0(Z)$ with its image in 
$\cl B(L^2(Z))$, so we abuse notation by writing $a$ in place of $M_a$.
We recall that the map $S_\varphi$ extends to a completely bounded map on $\cl K\otimes C_0(Z)$.
We observe that, when the identification (\ref{eq_KC}) is made, we have that the map 
$S_{\nph}$ (which is defined as a transformation on $\cl K\otimes C_0(Z)$)
is identified with $D_{\nph}$. 
Indeed, if $k\in L^2(Y\times X)$ and $a\in C_0(Z)$ then 
\[
    S_{\nph}(k\otimes a)(z) = (\nph(\cdot,\cdot,z)\cdot k) a(z) = D_{\nph}(k\otimes a)(z), \quad z\in Z.
\]
The equivalence now follows. 

(ii)$\Longleftrightarrow$(iii) 
is immediate from Lemma \ref{l_diagon}. 

(i)$\implies$(iv) Define a map $\psi : f \mapsto \psi_f$, on $L^1(Z)$ by letting
\[
	\psi_f (x,y) := \int_Z \varphi(x,y,z) f(z) dz , \quad (x,y) \in X \times Y .
\]
We will show that $\psi_f$ belongs to $L^\infty(X) \otimes^{w^*h} L^\infty(Y)$ and has norm at most $\| \varphi \|_\Sch$.
Take $f \in C_c(Z)$, $k \in L^2(Y \times X)$, and $a \in C_0(Z)$ with $\| a \| = 1$ and $a(z) = 1$ for all 
$z \in \supp(f)$. 
Writing $f = f_1 f_2$, $f_1,f_2 \in L^2(Z)$, $\| f\|_1 = \| f_1 \|_2 \| f_2 \|_2$, for $\xi \in L^2(X)$ and $\eta\in L^2(Y)$, we have
\[
\begin{split}
	\left| \ip{S_{\psi_f}(T_k)\xi}{\eta} \right| &= \left| \int_{X \times Y} \left( \int_Z \varphi(x,y,z) f(z) dz \right) k(y,x) \xi(x) \overline{\eta(y)} dx \, dy \right| \\
		&= \left| \ip{S_\varphi(T_{k \otimes a}) (\xi \otimes f_1)}{\eta \otimes f_2} \right| \\
		&\leq \| \varphi \|_\Sch \| T_{k \otimes a} \| \| \xi \|_2 \| f_1 \|_2 \| \eta \|_2 \| f_2 \|_2 \\
		&\leq \| \varphi \|_\Sch \| T_k \| \| f \|_1 \| \xi \|_2 \| \eta \|_2.
\end{split}
\]
Thus the map $S_{\psi_f}$ is bounded in the operator norm, implying that $\psi_f$ is a Schur multiplier
with $\| \psi_f \|_\Sch \leq \| \varphi \|_\Sch \| f \|_1$.
It follows from the density of $C_c(Z)$ in $L^1(Z)$ that $\psi$ is a bounded map, with $\| \psi \| \leq \| \varphi \|_\Sch$; we view $\psi$ as taking values in $L^\infty(X) \otimes^{w^*h} L^\infty(Y)$ using the standard identification of this tensor product with the Schur multipliers on $X \times Y$.

By standard operator space identifications (see \cite{CleMS} and \cite{lemerdy-tt}), 
we have
\[
	\psi \in \Bd[L^1(Z) , L^\infty(X) \otimes^{w^*h} L^\infty(Y)] \cong L^\infty(Z) \btens (L^\infty(X) \otimes^{w^*h} L^\infty(Y)),
\]
where $\varphi\in L^\infty(X\times Y\times Z)$ 
is the corresponding element in $L^\infty(Z) \btens (L^\infty(X) \otimes^{w^*h} L^\infty(Y))$.
Condition (iv) now follows by \cite[Theorem 19]{CleMS} and Theorem \ref{th_dia0}.

(iv)$\implies$(i) 
Let $v$ and $w$ be the functions arising as in Theorem \ref{th_dia0}, and
$M\subseteq X\times Y\times Z$ be a set with $(\mu\times\nu\times m)(M^c) = 0$, such that 
(\ref{eq_vw}) holds for all $(x,y,z) \in M$. 
Set $M_{x,y} = \{z: (x,y,z)\in M\}$ and $N = \{(x,y): m(M_{x,y}^c) = 0\}$; 
it is clear that $(\mu\times \nu)(N^c) = 0$.
Write $\mathcal{W}(y): L^2(Z)\to \HHH \otimes L^2(Z)$ and $\mathcal{V}(x): L^2(Z)\to \HHH \otimes L^2(Z)$ 
for the maps, given by
\[
    \big( \mathcal{V}(x) \xi \big)(z) := v(x,z) \xi(z) \text{ and } \big( \mathcal{W}(y) \xi \big)(z) := w(y,z) \xi(z), \quad \xi\in L^2(Z);
\]
we have
\begin{align*}
    \esssup_{x\in X} \| \mathcal{V}(x) \| &= \esssup_{(x,z) \in X\times Z} \| v(x,z) \| < \infty, \\
    \esssup_{y\in Y} \| \mathcal{W}(y) \| &= \esssup_{(y,z) \in Y\times Z} \| w(y,z) \| < \infty.
\end{align*}
For $a\in C_0(Z)$, $\xi,\eta \in L^2(Z)$ and $(x,y)\in N$, we have
\[
\begin{split}
    \left\langle \mathcal{W}(y)^* (I \otimes M_a) \mathcal{V}(x)\xi,\eta \right\rangle &=  
\left\langle (I \otimes M_a) \mathcal{V}(x)\xi,\mathcal{W}(y)\eta \right\rangle \\
        &= \int_Z a(z) \langle v(x,z),w(y,z)\rangle \xi(z)\overline{\eta(z)} dm(z)\\
        &= \int_Z a(z) \nph(x,y,z) \xi(z)\overline{\eta(z)} dm(z).
\end{split}
\]
It follows that, if $(x,y)\in N$ then
\[
    \mathcal{W}(y)^* (I \otimes M_a) \mathcal{V}(x) = M_{\nph_{x,y} a}, \quad a\in C_0(Z)
\]
(here $\nph_{x,y}$ is the function on $Z$ given by $\nph_{x,y}(z) = \nph(x,y,z)$). 
By \cite[Theorem 2.6]{mtt}, $\varphi$ is a Schur $C_0(Z)$-multiplier which is clearly central, and 
\[
    \| \varphi \|_\Sch \leq \esssup_{x \in X} \| \mathcal{V}(x) \| \esssup_{y \in Y} \| \mathcal{W}(y) \| 
    = \esssup_{z \in Z} \| \varphi_z \|_\Sch .
\]
 
Finally, from the proof of (i)$\implies$(ii)$\implies$(iii), equation (\ref{eq_ine}), 
and the estimate in (iv)$\implies$(i) we have 
$\| \varphi \|_\Sch = \sup_{z \in Z} \| \varphi_z \|_\Sch$.
\end{proof}

In the next result we assume that $A$ acts non-degenerately on a separable Hilbert space $\HH$, 
and we identify the elements of the centre $Z(\Mult(A))$ of $A$ with completely bounded maps on $A$ 
acting by operator multiplication. 

\begin{corollary}\label{th:centralmultsgeneral}
Let $\varphi : X \times Y \to Z(\Mult(A))$ be a pointwise measurable function, and assume that $\overline{ Z(A) A} = A$. 
The following are equivalent: 
\begin{enumerate}[i.]
    \item $\varphi$ is a central Schur $A$-multiplier;
    \item there exist an index set $I$ and operators
    $V \in C_I^\omega( L^\infty(X,Z(A)''))$ and $W \in C_I^\omega(L^\infty(Y ,Z(A)''))$, such that 
    \[
    	\varphi(x,y) = \sum_{i\in I} W_i(y)^* V_i(x), \quad \text{ for almost all $(x,y) \in X\times Y$.}
    \]
\end{enumerate}
Moreover, if $\varphi : X \times Y \to Z(\Mult(A))$ is weakly measurable then the above conditions are equivalent to:
\begin{enumerate}[i.] \setcounter{enumi}{2}
    \item $\varphi$ is a central Schur $B$-multiplier for any $C^*$-algebra $B \subseteq \Bd$ with $Z(A) \subseteq Z(B)$.
\end{enumerate}
If the conditions hold we may choose $V,W$ such that 
\[
	\| \varphi \|_\Sch = \| V \|_{C_I^\omega( L^\infty(X,Z(A)''))} \| W \|_{C_I^\omega( L^\infty(Y,Z(A)''))} , 
\]
where $\| \varphi \|_\Sch$ is the norm of the Schur multiplier in either (i) or (iii).
\end{corollary}
\begin{proof} 
Since $\overline{ Z(A) A} = A$, the algebra $Z(A)$ is non-degenerate and 
hence $Z(A)'' = \overline{Z(A)}^{w}$, where the closure in the weak operator topology.

\smallskip

(i)$\implies$(ii) 
By Remark~\ref{r_suba}, 
$\varphi$ is a Schur $Z(A)$-multiplier. 
Following the proof of Theorem~\ref{th_dia}, and using the identification  $Z(A) \cong C_0(Z)$ and $Z(A)'' \cong L^\infty(Z , m)$, for some measure space $(Z,m)$, we identify $\varphi$ with an element of 
$L^\infty(Z,m) \btens (L^\infty(X)\otimes^{w^*h} L^\infty(Y))$. 
Using \cite{CleMS}, 
we see that there exist an index set $I$ and two families 
$(V_i)_{i\in I}, \ (W_i)_{i\in I}$, where $V_i: X \to Z(A)''$ and $W_i: Y \to Z(A)''$ are measurable functions satisfying 
\[
	\esssup_{x\in X} \left\| \sum_{i\in I} V_i(x)^* V_i(x) \right\| <\infty \ \text{ and } 
	\ \esssup_{y\in Y} \left\|\sum_{i\in I} W_i(y)^* W_i(y) \right\| <\infty,
\] 
such that
$\varphi(x,y)=\sum_{i\in I} W_i(y)^* V_i(x)$ almost everywhere on $X\times Y$ (the series converges weakly) and 
\begin{equation}\label{eq_SchXY}
	\| \varphi \|_{\Sch(X,Y;Z(A))} = \esssup_{x\in X} \left\|\sum_{i\in I} V_i(x)^* V_i(x) \right\| 
	\esssup_{y\in Y} \left\| \sum_{i\in I} W_i(y)^* W_i(y) \right\| .
\end{equation}

(ii)$\implies$(i) For $a\in A$, we have
\begin{equation}\label{eq:centralmultanyB}
    \varphi(x,y)(a) = \sum_{i\in I} W_i(y)^*V_i(x) a = \sum_{i\in I} W_i(y)^* a V_i(x) = \mathcal{W}^*(y) \rho(a) \mathcal{V}(x) ,
\end{equation}
where $\mathcal{V}(x) \defeq (V_i(x))_{i \in I} , \ \mathcal{W}(y) \defeq (W_i(y)^*)_{i \in I}$ and $\rho(a) \defeq \id_{\ell^2(I)} \otimes a$.
By \cite[Theorem 2.6]{mtt} $\nph$ is a Schur $A$-multiplier, and it is clearly central.

\smallskip

(ii)$\implies$(iii) 
The assumption implies that $(x,y)\mapsto \varphi(x,y)(b) \in B$ is 
weakly measurable for all $b\in B$, so it makes sense to speak of $\varphi$ being a Schur $B$-multiplier.
Now the same proof as that of the implication (ii)$\implies$(i) can be applied. 

\smallskip

(iii)$\implies$(i) is trivial. 

\smallskip

For the norm equality observe that $\| \varphi \|_{\Sch(X,Y;B)} \geq \| \varphi \|_{\Sch(X,Y;Z(A))}$ while, 
by (\ref{eq:centralmultanyB}), we have 
\[
\begin{split}
	\| \varphi \|_{\Sch(X,Y;B)} &\leq \esssup_{x \in X} \| \mathcal{V}(x) \| \esssup_{y \in Y} \| \mathcal{W}(y) \| \\
		&= \esssup_{x\in X} \left\|\sum_{i\in I} V_i(x)^* V_i(x) \right\| \esssup_{x\in X} \left\|\sum_{i\in I} V_i(x)^* V_i(x)\right\| \\
		&= \| V \|_{C_I^\omega( L^\infty(X,Z(A)''))} \| W \|_{C_I^\omega( L^\infty(Y,Z(A)''))} .
\end{split}
\]
The equality follows by combining this with (\ref{eq_SchXY}).
\end{proof}

We remark that the results of this subsection and the rest of the section remain true when $X$ and $Y$ are discrete spaces with counting measures, $Z$ is an arbitrary (not necessarily second countable) locally compact Hausdorff space and $A$ is an arbitrary (not necessarily separable ) $C^*$-algebra. 

\subsection{Central Herz--Schur multipliers}
\label{ssec:centralHerzSchur}

In this subsection, similarly to Theorem \ref{th_dia}, we characterise central Herz--Schur multipliers, 
a natural invariant version of central Schur multipliers, which we now introduce.

\begin{definition}\label{de:centralHSmult}
Let $(A,G,\alpha)$ be a $C^*$-dynamical system.
A Herz--Schur $(A,G,\alpha)$-multiplier $F$ will be called \emph{central} if there exists a family $(a_r)_{r \in G} \subseteq Z(\Mult(A))$ such that
\[
    F(r)(a) = a_r a , \quad a \in A,\ r \in G .
\]
\end{definition}

\begin{proposition}\label{pr:restrictcentralHS}
Let $A$ be a $C^*$-algebra such that $\overline{Z(A) A} = A$, 
$(A, G,\alpha)$ be a $C^*$-dynamical system, 
$(a_r)_{r \in G}$ be a family in 
$Z(\Mult(A))$ and suppose that the map $F: G\to \cbm(A)$, given by $F(r)(a) = a_r a$, 
is pointwise measurable.  
The following are equivalent: 
\begin{enumerate}[i.]
   \item $F$ is a central Herz--Schur $(Z(A), G, \alpha)$-multiplier;
   \item $F$ is a central Herz--Schur $(A, G,\alpha)$-multiplier;
   \item there exist $V,W \in C_I^\omega(L^\infty(G , Z(A)''))$ such that
	\[
		\alpha_{t\inv}(a_{ts\inv}) = \sum_{i \in I} W_i(t)^* V_i(s), \text{ for almost all } (s,t)\in G\times G.
	\]
\end{enumerate}
Moreover, $V$ and $W$ may be chosen so that
\[
	\| F \|_\HS = \| V \|_{C_I^\omega( L^\infty(X,Z(A)''))} \| W \|_{C_I^\omega( L^\infty(Y,Z(A)''))} 
\] 
 where $\| F \|_\HS$ refers to the norm of $F$ in either (i) or (ii).
\end{proposition}
\begin{proof}
(i)$\implies$(ii) 
By \cite[Theorem 3.8]{mtt} $\mathcal{N}(F)$ is a Schur $Z(A)$-multiplier; it is clearly central.
Using the assumption $\overline{Z(A) A}= A$ we observe that $Z(A)$ acts non-degenerately on any Hilbert space where $A$ acts non-degenerately, so by Corollary~\ref{th:centralmultsgeneral} we have that $\mathcal{N}(F)$ is a central Schur $A$-multiplier. Applying again 
\cite[Theorem 3.8]{mtt}, we obtain that $F$ is a central Herz--Schur $(A, G,\alpha)$-multiplier.

(ii)$\implies$(i) Immediate from \cite[Theorem 3.8]{mtt} and Remark~\ref{r_suba}.

(i)$\implies$(iii) By \cite[Theorem 3.8]{mtt} $\mathcal{N}(F)$ is a central Schur $Z(A)$-multiplier, and for $a \in A$ and $s,t \in G$,
\[
	\mathcal{N}(F)(s,t)(a) = \alpha_{t\inv}(a_{ts\inv}) a, \quad a\in A.
\]
By Corollary~\ref{th:centralmultsgeneral}(ii), there exist $V,W \in C_I^\omega(L^\infty(G, Z(A)''))$ such that 
\[
	\alpha_{t\inv}(a_{ts\inv}) a = \sum_{i \in I} W_i(t)^* a V_i(s) = \sum_{i \in I} W_i(t)^* V_i(s) a  \quad \text{almost everywhere.}
\]
Since this holds for every $a \in A$ and $A \subseteq \Bd$ is separable and non-degenerate, 
we conclude that 
$$\alpha_{t\inv}(a_{ts\inv}) = \sum_{i \in I} W_i(t)^*V_i(s),$$ 
for almost all $(s,t) \in G \times G$.

(iii)$\implies$(i) For $a \in A$ and almost all $s,t \in G$ we have
\[
	\mathcal{N}(F)(s,t)(a) = \alpha_{t\inv}(a_{ts\inv}) a = \sum_{i \in I} W_i(t)^* a V_i(s) = \mathcal{W}(t)^* \rho(a) \mathcal{V}(s) ,
\]
where $\rho(a) := \id_{\ell^2(I)} \otimes a$, $\mathcal{V}(s) \defeq (V_i(s))_{i \in I}$ and $\mathcal{W}(t) \defeq (W_i(t))_{i \in I}$.
Therefore $F$ is a Herz--Schur $(Z(A),G,\alpha)$-multiplier by \cite[Theorem 3.8]{mtt}.

Since $\mathcal{N}$ is an isometry, the norm equality follows from the norm equality in Theorem~\ref{th:centralmultsgeneral}.
\end{proof}

A central Herz--Schur $(C_0(Z) , G ,\alpha)$-multiplier $F : G \to \cbm(C_0(Z))$,
associated with a family $(a_r)_{r \in G} \subseteq C_b(Z)$, 
may be identified with a bounded measurable function, continuous in the $Z$-variable, given by 
\[
	F : G \times Z \to \CC;\ F(r,z) = a_r(z) , \quad r \in G,\ z \in Z;
\]
conversely, 
if $F : G \times Z \to \CC$ is a bounded measurable function, continuous in the $Z$-variable, then the associated function $F : G \to \cbm(C_0(Z))$ is bounded and pointwise-measurable.
In the sequel, if $Z$ is a locally compact Hausdorff space and 
$(C_0(Z) , G , \alpha)$ is a $C^*$-dynamical system, we let
$(z,t)\to zt$ be the mapping from $Z\times G$ into $Z$ that satisfies the condition 
$f(zt) = \alpha_t(f)(z)$, $z\in Z$, $t\in G$. The mapping is jointly continuous and satisfies $z(st)=(zs)t$ for all $z\in Z$ and $s,t\in G$.

\begin{corollary}\label{co:centHSchar}
Let $(C_0(Z) , G , \alpha)$ be a $C^*$-dynamical system, and $F : G \times Z \to \CC$ a bounded measurable function, continuous in the $Z$-variable. 
The following are equivalent: 
\begin{enumerate}[i.]
    \item $F$ is a  central Herz--Schur $(C_0(Z) , G , \alpha)$-multiplier;
    \item there exists a Hilbert space $\HHH$ and weakly measurable bounded functions $v,w : G \times Z\to \HHH$ such that 
 	\[
 		F(ts\inv , zt^{-1}) = \ip{v(s,z)}{w(t,z)} \quad \text{almost all $(s,t,z) \in G\times G\times Z$.}
 	\]
\end{enumerate}
Moreover, 
$\| F \|_\HS = \esssup_{(s,x) \in G \times Z}\| v(s,x) \| \esssup_{(t,y) \in G \times Z} \| w(t,y) \|$.
\end{corollary}
\begin{proof}
Immediate from Proposition~\ref{pr:restrictcentralHS} by taking $\HHH := \ell^2(I)$, 
\[
	v(s,x)_i := \big( V_i(s) \big)(x) \quad \text{ and } \quad w(t,y)_i := \big( W_i(t) \big)(y) , \quad s,t \in G,\ x,y \in Z .
\]
\end{proof}

\subsection{Positive central multipliers}
\label{ssec:poscentral}
Positive Schur $A$-multipliers, 
in the case of sets equipped with the counting measure, 
were studied in \cite{mstt} (see \cite[Definition 2.3]{mstt} and \cite[Theorem 2.6]{mstt}).
Here we extend this by considering arbitrary standard measure spaces and identifying 
corresponding versions of the previous results.

\begin{definition}\label{de:posmeasurableSchurmult}
Let $A$ be a $C^*$-algebra.
A Schur $A$-multiplier $\varphi : X \times X \to {\rm CB}(A)$ is called \emph{positive} if $S_\varphi$ is completely positive.
\end{definition}

Before giving a completely positive version of Theorem \ref{th_dia}, we include a lemma. 
Since $L^{\infty}(X)\otimes^{w^* h} L^{\infty}(X) = (L^1(X)\otimes^h L^1(X))^*$, every Schur multiplier
$\nph$ on $X\times X$ gives rise to a canonical bilinear map $F_{\nph} : L^1(X)\times L^1(X)\to \mathbb{C}$.
As usual, we write $F_{\nph}^{(n,n)}$ for the corresponding amplification, a bilinear map from $M_n(L^1(X))\times M_n(L^1(X))$ 
into $M_n$.

\begin{lemma}\label{l_psm}
Let $(X,\mu)$ be a standard measure space and $\nph \in L^{\infty}(X)\otimes^{w^* h} L^{\infty}(X)$ 
be a positive Schur multiplier.
If $T = (f_{i,j})_{i,j=1}^n \in M_n(L^1(X))$ and $T^* = (\overline{f_{j,i}})_{i,j=1}^n$ then 
$F_{\nph}^{(n,n)}(T,T^*) \in M_n^+$.
\end{lemma}

\begin{proof}
Note that, if $\nph$ is a positive Schur multiplier, by virtue of \cite{Ha80}, one may write
$\nph = \sum_{i=1}^{\infty} a_i \otimes \overline{a_i}$, where $(a_i)_{i=1}^{\infty}$ is a bounded row operator 
with entries in $L^{\infty}(X)$. 
It thus suffices to prove the statement in the case where $\nph = a\otimes \overline{a}$, for some $a\in L^{\infty}(X)$.
However, then we have
\[
    F_{\nph}^{(n,n)}(T,T^*) = \left(\sum_{k=1}^n \langle f_{i,k},a\rangle \langle \overline{f_{j,k}},\overline{a}\rangle\right)_{i,j=1}^n = \sum_{k=1}^n \left(\langle f_{i,k},a\rangle \overline{\langle f_{j,k},a\rangle}\right)_{i,j=1}^n ,
\]
and the conclusion follows. 
\end{proof}

\begin{theorem}\label{th_poscentral}
Let $\varphi : X \times X \times Z \to \CC$ be a bounded measurable function, continuous in the $Z$-variable. 
The following are equivalent: 
\begin{enumerate}[i.]
    \item $\varphi$ is a positive central Schur $C_0(Z)$-multiplier;
    \item there exists a Hilbert space $\HHH$ and an essentially bounded, weakly measurable function $v : X \times Z \to \HHH$ such that $\varphi(x , y ,z) = \ip{v(x,z)}{v(y,z)}$ for almost all $(x,y,z) \in X \times X\times Z$ ; 
    \item for each $z \in Z$ the function $\varphi_z$ is a positive Schur multiplier, and
    \[
        \sup_{z\in Z} \|\nph_z\|_\Sch < \infty .
    \]
\end{enumerate}
Moreover, if the space $X$ is discrete and $\mu$ is counting measure the above conditions are equivalent to:
\begin{enumerate}[i.] \setcounter{enumi}{3}
    \item for any $x_1 , \ldots , x_n \in X$ and $z \in Z$ the matrix $(\varphi(x_i , x_j , z))_{i,j}$ is positive in $M_n$. 
\end{enumerate}
\end{theorem}

\begin{proof}
(i)$\implies$(ii) 
Suppose that $\varphi$ is a positive central Schur $C_0(Z)$-multipli\-er. 
We have seen in the proof of Theorem~\ref{th_dia} that $\varphi \in L^\infty(Z) \btens (L^\infty(X)\otimes^{w^*h} L^\infty(X))$. 
With $\varphi$ we associate the completely bounded bilinear 
map $\Phi_\varphi : L^1(X)\times L^1(X) \to L^\infty(Z)$ given by 
\[
    \Phi_\varphi((f,g))(h) = \ip{\varphi}{h \otimes (f\otimes g)} , \quad f,g \in L^1(X),\ h \in L^1(Z) .
\]
 We obtain
\begin{equation}\label{eq:SphiactingonHaageruptp}
\begin{split}
    \Phi_\varphi((f,g))(h) &= \iiint_{X \times X \times Z} \varphi(x,y,z) h(z)f(x)g(y) dx\, dy\, dz \\
    &= \int_Z\left(\int_{X\times X}\varphi_z(x,y)f(x)g(y)dxdy\right)h(z)dz
\end{split}
\end{equation}
and  
$$\Phi_\varphi((f,g))(z) = \int_{X\times X}\varphi_z(x,y)f(x)g(y)dxdy \ \  \text{ a.e.}.$$
Set 
$$\Phi_{\varphi_z}((f,g)) = \Phi_\varphi((f,g))(z), \ \ z\in Z.$$ 
By Lemma \ref{l_diagon}, $\varphi_z$ is a positive Schur multiplier
and, by Lemma \ref{l_psm}, 
$\Phi_{\varphi_z}^{(n,n)}(((f_{i,j}), (f_{i,j}^*)))\in M_n^+$ for any $(f_{i,j})\in M_n(L^1(X))$.
By \cite[Theorem 4.4, Remark 4.5(iii)]{ss}, there exists a family 
$(\psi_i)_{i \in \Lambda} \subseteq \cbm(L^1(X) , L^\infty(Z))$ such that
$\|\sum_{i\in I}|\psi_i(a)|^2\|_\infty\leq C\|a\|_1^2$, $a\in L^1(X)$, for some constant $C>0$, and 
$$\Phi_\varphi ((a,b)) = \sum_{i \in \Lambda} \psi_i(a) \psi_i(b^*)^*, \ \ \ a,b \in L^1(X).$$
Identifying each $\psi_i$ with an element $\psi_i$ of $L^\infty(X \times Z)$ \textit{via}
\[
	\psi_i(f)(h) = \int_X \int_Z \psi_i(x,z) f(x) h(z) dx \, dz , \quad f \in L^1(X),\ h\in L^1(Z) ,
\]
letting $\HHH = \ell^2(\Lambda)$ and $v(x,z) := (\psi_i(x,z))_{i \in \Lambda}$ gives (ii).

(ii)$\implies$(i)  Define
\[
	\mathcal{V}(x) : L^2(Z) \to \HHH \otimes L^2(Z) ;\ \big( \mathcal{V}(x) \xi \big) (z) := v(x,z) \xi(z) , \quad \xi \in L^2(Z) .
\]
Then
\[
    \varphi(x,y)(a) = \mathcal{V}(y)^* \big( \id \otimes M_a \big) \mathcal{V}(x) , \quad a \in C_0(Z)  
\]
for almost all $(x,y)$ (see the proof of Theorem \ref{th_dia} (iv)$\iff$(i)).
Therefore $\varphi$ is a central Schur $C_0(Z)$-multiplier, and (as in the proof of \cite[Theorem 2.6]{mtt}) writing $\rho$ for the representation $a\mapsto \id \otimes M_a$ of $C_0(Z)$ on $\HHH \otimes L^2(Z)$ we have
\[
	S_\varphi(T) = \mathcal{V}^* (\id \otimes \rho)(T) \mathcal{V} , \quad T \in \Cpt[L^2(X)] \otimes C_0(Z) .
\] 
Hence $S_\varphi$ is completely positive.

(i)$\iff$(iii) follows from the following two facts:
(a) since $\varphi$ is a Schur $C_0(Z)$-multiplier, we have that 
$S_\varphi(K)(z)=S_{\varphi_z}(K(z))$, $z\in Z$ for any $K\in C_0(Z,\cl K)$, and 
(b) an element $K\in C_0(Z,\cl K)$ is positive if and only if 
$K(z)\geq 0$ as an operator in $\cl K$ for all $z\in Z$.  

Now assume that $\mu$ is counting measure on the discrete space $X$.
Observe that (iv) is equivalent to $(\varphi(x_i , x_j))$ being a positive element of $M_n(C_0(Z))$.

(i)$\implies$(iv) Let $x_1 , \ldots , x_n \in X$. 
By \cite[Theorem 2.6]{mstt}, 
the matrix $(\varphi(x_i , x_j)(a)) \in M_n(C_0(Z))$ is positive when $a \in C_0(Z)$ is positive. 
For a fixed $z_0\in Z$, let $a\in C_0(Z)$ be such that $a(z_0) = 1$.
It follows that $(\varphi(x_i , x_j, z_0))_{i,j} \in M_n^+$.

(iv)$\implies$(i) For a positive $(a_{i,j}) \in M_n(C_0(Z))$, the matrix $(\varphi(x_i , x_j)( a_{i,j}))$ is the Schur product of $(\varphi(x_i , x_j))$ and $(a_{i,j})$ in $M_n(C_0(Z))$. Since (iv) ensures the positivity of $(\varphi(x_i , x_j))$, and the Schur product of two positive matrices over a commutative $C^*$-algebra is positive, (i) follows from \cite[Theorem 2.6]{mstt}.
%
\end{proof}

In the next corollary we assume $A$ acts nondegenerately on a separable Hilbert space $\HH$.

\begin{corollary}\label{co:poscentgeneralcase}
Let $\varphi : X \times X \to Z(\Mult(A))\subseteq \cbm(A)$ be a pointwise measurable function, and assume that $\overline{ Z(A) A} = A$. 
The following are equivalent: 
\begin{enumerate}[i.]
    \item $\varphi$ is a positive central Schur $A$-multiplier;
    \item there exist an index set $I$ and $V \in C_I^\omega( L^\infty(X,Z(A)''))$ such that 
    \[
    	\varphi(x,y) = \sum_{i\in I} V_i(y)^* V_i(x), \quad \text{ for almost all $(x,y) \in X\times Y$.}
    \]
\end{enumerate}
Moreover, if $\varphi : X \times X \to Z(\Mult(A))$ is weakly measurable then the above conditions are equivalent to:
\begin{enumerate}[i.] \setcounter{enumi}{2}
    \item $\varphi$ is a positive central Schur $B$-multiplier for any $C^*$-algebra $B \subseteq \Bd$ with $Z(A) \subseteq Z(B)$.
\end{enumerate}
\end{corollary}
\begin{proof}
Follows from Theorem~\ref{th_poscentral} in the same way as Corollary~\ref{th:centralmultsgeneral} follows from Theorem~\ref{th_dia}.
\end{proof}

We recall the following definition from \cite{mstt}.

\begin{definition}\label{de:posHS}
A Herz--Schur $(A, G,\alpha)$-multiplier $F:G\to \cbm(A)$ is called \emph{completely positive} if $S_{F}$ is completely positive on $A\rtimes_{\alpha, r} G$.
\end{definition}

\begin{theorem}\label{th:poscentraltransference}
Let $(A, G, \alpha)$ be a $C^*$-dynamical system such that $\overline{Z(A)A} = A$, 
and $F: G\to Z(\Mult(A))$ be a pointwise measurable function. 
The following are equivalent:
\begin{enumerate}[i.]
	\item $F$ is a completely positive central Herz--Schur $(Z(A), G,\alpha)$-multiplier;
	\item $F$ is a completely positive central Herz--Schur $(A, G,\alpha)$-multiplier;
	\item $\mathcal{N}(F)$ is a positive central Schur $Z(A)$-multiplier;
	\item $\mathcal{N}(F)$ is a positive central Schur $A$-multiplier.
\end{enumerate}
\end{theorem}
\begin{proof}
%
(ii)$\implies$(iv) Assume that $F: G\to Z(\Mult(A))$ is a positive central Herz--Schur $(A, G,\alpha)$-multiplier. 
By the proof of \cite[Theorem 3.8]{mtt}, 
using the Stinespring dilation theorem in place of the Haagerup--Paulsen--Wittstock theorem, 
we have $S_F(T) = V^* \rho(T) V$, $T\in A \rtimes_{\alpha,r} G$. 
The representation $\rho \circ (\pi \rtimes \lambda)$ of the full crossed product $A \rtimes_\alpha G$ 
has the form $\rho_A \rtimes \rho_G$, where $(\rho_A,\rho_G)$ is a covariant pair.
Let $\mathcal{V}(s) := \rho_{G}(s^{-1})V\lambda_s$; 
as in \cite[page 408]{mtt}, we have 
$\mathcal{N}(F)(s,t)(a) = \mathcal{V}(t)^* \rho_A(a) \mathcal{V}(s)$, 
so $S_{\mathcal{N}(F)} = \mathcal{V}^*(\rho_A \otimes \id)(\cdot) \mathcal{V}$ is completely positive.
Therefore $\mathcal{N}(F)$ is a positive Schur $A$-multiplier, and it is clearly central.

(iv)$\implies$(ii) 
As in the proof of \cite[Theorem 3.8]{mtt} we have that 
$S_{F} = S_{\mathcal{N}(F)}|_{A\rtimes_{\alpha, r} G}$, so $S_F$ is completely positive.

(iv)$\implies$(iii) Follows from Remark~\ref{r_suba}.

(iii)$\implies$(iv) Let $\mathcal{N}(F)$ be a positive central Schur $Z(A)$-multiplier. 
Following the proof of the implication (i)$\implies$(ii) of 
Corollary~\ref{th:centralmultsgeneral} and applying \cite[Remark 4.5(iii)]{ss}, 
we see that 
there exists an index set $I$ and an essentially bounded function $V \in C_I^\omega(L^\infty(G , Z(A)''))$ 
such that $\mathcal{N}(F)(s,t) = \sum_{i \in I} V_i(t)^* V_i(s)$ almost everywhere on $G\times G$ (the series converges weakly). 
Hence for $a\in A$ and $s,t \in G$ we have 
\[
	\mathcal{N}(F)(s,t)(a) = \sum_{i\in I} V_i(t)^* V_i(s) a = \sum_{i\in I} V_i(t)^* a V_i(s) = \mathcal{V}(t)^* \rho(a) \mathcal{V}(s) ,
\]
where $\mathcal{V}(r) := (V_i(r))_{i \in I}$ and $\rho(a) = \id \otimes a$.
As in the proof of the implication (ii)$\implies$(i) of \cite[Theorem 2.6]{mtt}, 
it follows that $S_{\mathcal{N}(F)} = \mathcal{V}^* (\id \otimes \rho)(\cdot) \mathcal{V}$ is completely positive, so $\mathcal{N}(F)$ is a positive central Schur $A$-multiplier.

(i)$\iff$(iii) This is a special case of (ii)$\iff$(iv).
%
\end{proof}

Using Theorem \ref{th_poscentral}, similarly to Corollary \ref{co:centHSchar}, one can obtain the following 
description of completely positive central Herz--Schur $(C_0(Z) , G , \alpha)$-multipliers. 

\begin{corollary}\label{co:poscentHSchar}
Let $(C_0(Z) , G , \alpha)$ be a $C^*$-dynamical system, and $F : G \times Z \to \CC$ a measurable function, continuous in the $Z$-variable. 
The following are equivalent:
\begin{enumerate}[i.]
    \item $F$ is a completely positive central Herz--Schur $(C_0(Z) , G , \alpha)$-multiplier;
    \item there exists a Hilbert space $\HHH$ and a weakly measurable function $v : G \times Z\to \HHH$  such that  $F(ts\inv ,xt\inv) = \ip{v(s,x)}{v(t,x)}$ almost everywhere on $G\times G\times Z$.
\end{enumerate}
\end{corollary}


%
\subsection{Connections with other types of multipliers}
\label{ssec:connectionscentral}

Let $Z$ be a locally compact Hausdorff space, equipped with an action of a locally compact group $G$;
thus, we are given a map 
$Z\times G\to Z$, $(x,s)\to x s$, jointly continuous and such that $x (st) = (xs) t$ for all $x\in Z$ and all $s,t\in G$.
We consider the crossed product $C_0(Z)\rtimes_{\alpha,r}G$, 
where $\alpha$ is the corresponding action of $G$ on $C_0(Z)$.
The set $\cl G = Z\times G$ is a groupoid, where the set $\cl G^2$ of composable pairs is given by 
$$\cl G^2 = \{[(x_1,t_1),(x_2,t_2)]: x_2 = x_1t_1\},$$
and if $[(x_1,t_1),(x_2,t_2)] \in \cl G^2$, the product
$(x_1,t_1)\cdot (x_2,t_2)$ is defined to be $(x_1,t_1t_2)$,
while the inverse $(x,t)^{-1}$ of $(x,t)$ is defined to be $(xt, t^{-1})$.
The domain and range maps are given by
\[
    d((x,t)) := (x,t)^{-1}\cdot(x,t)=(xt,e), \quad r((x,t)) := (x,t)\cdot(x,t)^{-1}=(x,e).
\]
The unit space $\cl G_0$ of the groupoid, which is by definition equal to the common image of the maps $d$ and $r$, can therefore be canonically identified with $X$.
We refer to \cite{Ren80} for background on groupoids (see also \cite[Section 5.2]{mtt}). 

Let $\psi : Z \times G \to \CC$ be a bounded continuous function.
Let $F_\psi(s)\in CB(C_0(Z))$ be given by $F_\psi(s)(f)(x) := \psi(x,s)f(x)$, $f\in C_0(Z)$, $s\in G$. 
In \cite[Section 5]{mtt} it was shown that such a function $\psi$ is a Herz--Schur $(C_0(Z) , G , \alpha)$-multiplier if and only if $\psi$ is a completely bounded multiplier of the Fourier algebra of $\cl{G}$ in the sense of Renault~\cite{Ren97}.
In the terminology of this paper such functions $\psi$ are central Herz--Schur $(C_0(Z) , G , \alpha)$-multipliers.
The following is therefore immediate from \cite[Proposition 5.3]{mtt} and Corollary~\ref{co:centHSchar}.

\begin{corollary}\label{co:centralmultscharMTT}
Let 
$(C_0(Z) , G , \alpha)$ be a $C^*$-dynamical system, and write $\cl{G}$ for the underlying groupoid. 
Let $\psi : Z \times G \to \CC$ be a bounded continuous function and write $F_\psi (r)(f)(x) := \psi(x,r)f(x)$, $f\in C_0(Z)$.
The following are equivalent:
\begin{enumerate}[i.]
    \item $F_\psi$ is a central Herz--Schur $(C_0(Z) , G , \alpha)$-multiplier;
    \item $\psi$ is a completely bounded multiplier of the Fourier algebra of $\cl{G}$;
    \item there exist a Hilbert space $\HHH$ and essentially bounded functions $v , w : G \times Z \to \HHH$ such that
    \[
        \psi(x t\inv , ts\inv) = \ip{v(s,x)}{w(t,x)} , \quad s,t \in G ,\ \text{almost all $x \in X$} .
    \]
\end{enumerate}
If the conditions hold then we can choose $v$ and $w$ such that  
\[
	\| \psi \|_\HS = \esssup_{(s,x) \in G \times Z} \| v(s,x) \| \esssup_{(t,x) \in G \times Z} \| w(t,x) \| .
\]
\end{corollary}

We next link central multipliers to the multipliers studied by Dong--Ruan in \cite{DR12}.
Let $(A,G,\alpha)$ be a $C^*$-dynamical system with $A$ unital and $G$ discrete. Dong--Ruan define a function $h : G \to A$ to be a \emph{multiplier with respect to $\alpha$} if there is an $A$-bimodule map $\Phi$ on $A \rtimes_{\alpha ,r} G$ such that $\Phi(\lambda_r) = \lambda_r \pi(h(r))$. 
The $A$-bimodule requirement forces $h(r) \in Z(A)$ for all $r \in G$. Hence $\Phi=S_F$ for the central $(A,G,\alpha)$ multiplier given by $F(r)(a)=h(r)a$.

In \cite[Section 6]{DR12}, the authors use the fact 
that classical (positive) Schur multipliers on a discrete group $G$ give rise to (positive) central Herz--Schur multipliers of $(\ell^\infty(G) , G , \beta)$ (here $\beta$ denotes the left translation action). This connection is also 
utilised by Ozawa~\cite{Oz00}.
We formalise this connection in the next proposition. 

\begin{proposition}\label{pr:centralmultsareSchur}
Let $G$ be a discrete group. 
Consider a function $\varphi : G \times G \to \CC$ and a family $a = (a_r)_{r \in G} \subseteq C_b(G)$.
Define
\[
    a^\varphi_r(p) := \varphi(r\inv p\inv , p\inv) \quad \text{ and } \quad \varphi_a(s,t) := a_{ts\inv}(t\inv) .
\]
The assignments $\varphi \mapsto a^\varphi$ and $a \mapsto \varphi_a$ are mutual inverses, and give a one-to-one correspondence between the classical Schur multipliers and the central Herz--Schur $(C_0(G) , G , \beta)$-multipliers.
This bijection is an isometric algebra isomorphism which preserves positivity. 
\end{proposition}
\begin{proof}
It is easy to check that $\varphi_{a^\varphi} = \varphi$ and $a^{\varphi_a} = a$ and that these assignments are linear and multiplicative.

Now suppose that $a =(a_r)_{r \in G}$ is a central Herz--Schur $(C_0(G) , G , \beta)$-multiplier.
By Corollary~\ref{co:centHSchar}, there exists a Hilbert space $\mathcal{L}$ and weakly measurable 
functions $v,w : G\times G\to \mathcal{L}$, such that 
\[
    \varphi_a(s,t) = a_{ts\inv}(t\inv) = \ip{v(s,e)}{w(t,e)}, \quad s,t\in G.
\] 
It follows from \cite{BoFe84} 
that $\varphi_a$ is a Schur multiplier and $\| \varphi_a \|_\Sch \leq \| a \|_\HS$.

Conversely, suppose $\varphi : G \times G \to \CC$ is a Schur multiplier, 
and take $v,w : G \to \HH$ are such that 
$\varphi(s,t) = \ip{v(s)}{w(t)}$ and $\| \varphi \|_\Sch$ $=$ $\sup_{s \in G}\| v(s) \|$ $\sup_{t \in G} \| w(t) \|$.
Then, for $s,t,x \in G$,
\[
    a^\varphi_{ts\inv}(x t\inv) = \varphi(s t\inv t x\inv , t x\inv) = \varphi(s x\inv , t x\inv) = \ip{v(s x\inv)}{w(t x\inv)} ,
\]
so, by Corollary~\ref{co:centHSchar}, 
$a^\varphi = (a^\varphi_r)_{r \in G}$ is a central Herz--Schur $(C_0(G) , G , \alpha)$-multiplier with 
$\| a^\varphi \|_\HS \leq \| \varphi \|_\Sch$.


If $a$ is a positive central multiplier (resp.\ $\varphi$ is a positive Schur multiplier) then applying Corollary~\ref{co:poscentHSchar}, taking $v = w$ in the above calculations, shows $\varphi_a$ (resp.\ $a^\varphi$) is also positive.
\end{proof}

\section{Convolution multipliers}
\label{sec:convmults}

In this section, we give a characterisation of Herz--Schur convolution multipliers first studied in \cite[Section 6]{mtt}.
We will use the notion of a Herz--Schur $\theta$-multiplier of a $C^*$-dynamical system $(A,G,\alpha)$,
introduced in \cite[Definition 3.3]{mtt}. 
Let $\theta : A \to \Bd[\HH_\theta]$ be a faithful representation of (the separable $C^*$-algebra) 
$A$ on the separable Hilbert space $\HH_\theta$, 
and let $(\pi^\theta , \lambda^\theta)$ be the regular covariant pair associated to this representation
(see Subsection~\ref{sss_cp}). 
A function $F : G \to \cbm(A)$ will be called a \emph{Herz--Schur $\theta$-multiplier of $(A,G,\alpha)$} if the map
\[
	\pi^\theta(a) \lambda^\theta_r \mapsto \pi^\theta\big( F(r)(a) \big) \lambda^\theta_r
\]
extends to a completely bounded, weak*-continuous map on $A \rtimes^{w^*}_{\alpha , \theta} G$. 
As before we assume that $G$ is either second countable or discrete.

\subsection{Abelian case}
\label{ssec:abelianconvmults}

Let $G$ be an abelian locally compact group  
equipped with a Haar measure and $\Gamma$ be its dual group. 
We denote by $\lambda^{\Gamma}$ the left regular representation on $L^2(\Gamma)$. 
We shall identify each element $s\in G$ with a character on $\Gamma$, and use $\beta$ to denote the natural action of $G$ on $C_r^*(\Gamma)$ by letting
\[
    \beta_s \big( \lambda^\Gamma(f) \big) := \lambda^\Gamma(sf) , \quad  s\in G,\ f \in L^1(\Gamma);
\]
thus, $(C^*_r(\Gamma), G,\beta)$ is a $C^*$-dynamical system. 

Given a bounded measurable function $\psi : G\times\Gamma \to \CC$ and $t\in G$ (resp.\ $x\in \Gamma$), let the function $\psi_t : \Gamma \to \CC$ (resp.\ $\psi^x : G \to \CC$) be given by $\psi_t(y) := \psi(t,y)$ (resp. $\psi^x(s) := \psi(s,x)$).
We call $\psi$ \emph{admissible} if $\psi_t\in B(\Gamma)$ for every $t\in G$ and $\sup_t \|\psi_t\|_{B(\Gamma)} < \infty$.
Assuming that $\psi$ is admissible, let $F_\psi(t) : C_r^*(\Gamma)\to C_r^*(\Gamma)$ be the map given by
\[
    F_\psi(t)(\lambda^\Gamma(g)) = \lambda^\Gamma(\psi_t g), \quad g\in L^1(\Gamma) .
\]
We define the \emph{Herz--Schur convolution multipliers of $G$} to be the elements of the set
\[
\begin{split}
    \HSconv(G) := \{ \psi : G\times \Gamma \to \CC &\ : \psi \text{ is admissible and }  F_\psi \text{ is }\\
        & \text{ a Herz--Schur $(C^*_r(\Gamma),G,\beta)$-multiplier} \},
 \end{split}
\]
and write
\[
\begin{split}
    \HSconv^\id (G) = \{\psi : G\times \Gamma \to \CC &\ : \psi \text{ is admissible and } F_\psi \text{ is }\\
    	& \text{ a Herz--Schur $\id$-multiplier of $(C^*_r(\Gamma),G,\beta)$} \}.
\end{split}
\]
Here we write $\id$ for the canonical representation of $C^*_r(\Gamma)$ on $L^2(\Gamma)$.
Clearly, the space $\HSconv(G)$ is an algebra with respect to the operations of pointwise addition and multiplication, and $\HSconv^\id(G)$ is a subalgebra of $\HSconv(G)$.
For $\psi\in \HSconv(G)$, let $\|\psi\|_\HS = \|F_{\psi}\|_\HS$, and use $S_{\psi}$ to denote the map $S_{F_{\psi}}$.

We identify an elementary tensor $u\otimes h$, where $u\in B(G)$ and $h\in B(\Gamma)$, with the function $(s,x)\to u(s)h(x)$, $s\in G$, $x\in \Gamma$.
Let $\mathfrak{F}(B(G),B(\Gamma))$ be the complex vector space of all separately continuous functions $\psi : G\times \Gamma \to \CC$ such that, for every $s\in G$ (resp.\ $x\in \Gamma$), the function $\psi_s : \Gamma \to \CC$ (resp.\ $\psi^x : G \to \CC$) belongs to $B(\Gamma)$ (resp.\ $B(G)$).
By \cite[Section 6]{mtt}, we have the following inclusions: 
\[
    B(G)\odot B(\Gamma)\subseteq \HSconv^\id(G) \subseteq \mathfrak{F}(B(G),B(\Gamma)) .
\]
We now answer \cite[Question 6.6]{mtt} by identifying $\HSconv^\id(G)$.

\begin{theorem}\label{conv_ab}
Let $G$ be a locally compact abelian group 

and $\psi : G \times \Gamma \to \CC$ be an admissible function.
The following are equivalent:
\begin{enumerate}[i.]
	\item $\psi \in \HSconv^\id(G)$;
	\item $\psi \in B(G\times\Gamma)$.
\end{enumerate}
The identification is an isometric algebra homomorphism.
\end{theorem}
\begin{proof}
(i)$\implies$(ii) 
Let $\psi \in \HSconv^\id(G)$ and let $F_\psi : G \to \cbm(C^*_r(\Gamma))$ be the corresponding Herz-Schur multiplier of $(C^*_r(\Gamma),G,\beta)$. 
By \cite[Theorem 3.8]{mtt}, 
$\cl N(F_\psi)(s,t)$ is a Schur $C^*_r(\Gamma)$-multiplier and hence there exists a Hilbert space $\HH_\rho$, operators $\mathcal{V}, \mathcal{W} \in L^\infty(G,\Bd[L^2(\Gamma),\HH_\rho])$, a continuous 
unitary representation $\rho: \Gamma \to \Bd[\HH_\rho]$ 
and a subset $N\subseteq G\times G$ with $(m_G\times m_G)(N) = 0$,
such that
\[
    \cl N(F_\psi)(s,t) \big( \lambda^\Gamma(f) \big) = \mathcal{W}(t)^* \rho(f) \mathcal{V}(s), \quad f \in L^1(\Gamma),
\]
for all $(s,t) \not\in N$, and  
\begin{equation}\label{eq_noess}
    \|\psi\|_\Sch = \esssup_{s\in G} \| \mathcal{V}(s) \| \esssup_{t\in G} \| \mathcal{W}(t) \|.
\end{equation}
As 
\[
    \cl N(F_\psi)(s,t)(\lambda^\Gamma(f)) = \beta_{t^{-1}}(F_\psi(ts^{-1})(\beta_t(\lambda^\Gamma(f)))) = \lambda^\Gamma(\psi_{ts^{-1}}f),
\]
we obtain
\[
    \lambda^\Gamma(\psi_{ts^{-1}}f) = \mathcal{W}(t)^* \rho(f) \mathcal{V}(s), \quad f\in L^1(\Gamma) ,
\]
for all $(s,t)\notin N$.
As $\psi_{ts^{-1}}\in B(\Gamma)$, 
we have that $\psi_{ts^{-1}}$ 
is a completely bounded multiplier of $A(\Gamma)$, and the map $S_{\psi_{ts^{-1}}}$ can be extended to a 
weak*-continuous linear operator on $\vN(\Gamma)$; we have
\[
    \psi(ts^{-1},x)\lambda_x^\Gamma = \mathcal{W}(t)^* \rho(x) \mathcal{V}(s), \quad x \in \Gamma, \ (s,t)\notin N .
\]
Thus, for $\xi\in L^2(\Gamma)$ with $\|\xi\|_2=1$, we have 
\[
\begin{split}
    \psi(ts^{-1}, xy^{-1}) \ip{\xi}{\xi} 
    &= \ip{\lambda_{x^{-1}}^\Gamma \mathcal{W}(t)^* \rho(x) \rho(y)^* \mathcal{V}(s) \lambda_y^\Gamma\xi}{\xi} \\
        &= \ip{\rho(y)^* \mathcal{V}(s) \lambda_y^\Gamma \xi}{\rho(x)^* \mathcal{W}(t) \lambda_x^\Gamma \xi} .
\end{split}
\]
Letting 
$v(s,y) :=\rho(y)^* \mathcal{V}(s)\lambda_y^\Gamma\xi$ and 
$w(t,x) := \rho(x)^* \mathcal{W}(t)\lambda_{x}^\Gamma\xi$, we obtain
\[
    \psi \big((t,x)(s,y)^{-1} \big) = \ip{v(s,y)}{w(t,x)} , \ \ \ (s,t)\notin N.
\]
By \cite{BoFe84},
$\psi$ is equal almost everywhere to a completely bounded multiplier of $A(G\times \Gamma)$, 
and hence to an element $u\in B(G\times \Gamma)$ \cite[Theorem 5.1.8]{KaLa18}. 
To see that $\psi(t,x)=u(t,x)$ for all $(t,x)$, for each $t\in G$ we let 
$$N_t = \{x\in \Gamma : \psi(t,x) = u(t,x)\}.$$ 
By Fubini's Theorem, the set $\{t\in G: m_\Gamma(N_t^c)> 0\}$ has measure zero, that is, 
for almost all $t\in G$, we have that 
$\psi(t,x) = u(t,x)$ almost everywhere. As $\psi$ is separately continuous, the last equality holds for all $x\in\Gamma$. 
Using again the separate continuity of $\psi$ we obtain that $\psi(t,x)=u(t,x)$ for all $(t,x)$. 
Furthermore, by (\ref{eq_noess}),
\[
\begin{split}
    \|\psi\|_{B(G\times \Gamma)} 
    &\leq \esssup_{(s,y)\in G\times \Gamma} \|\rho(y)^* \mathcal{V}(s) \lambda_{y}^\Gamma \xi\| \esssup_{(t,x)\in G\times \Gamma} \| \rho(x)^* \mathcal{W}(t)\lambda_{x}^\Gamma \xi \| \\
        &\leq \esssup_{s\in G} \| \mathcal{V}(s) \| \esssup_{t\in G} \| \mathcal{W}(t) \| = \| \psi \|_\Sch.
    \end{split}
\]

(ii)$\implies$(i) Assume that $\psi\in B(G \times \Gamma)$. 
By \cite{BoFe84}, 
there exist a Hilbert space $\HH$ and continuous $v,w : G \times \Gamma \to \HH$ such that 
$$\psi(ts^{-1}, xy^{-1}) = \ip{v(s,y)}{w(t,x)}, \ \ \ s,t\in G, x,y\in \Gamma,$$ 
and 
$$\| \psi \|_{B(G\times \Gamma)} = \sup_{(s,y)} \| v(s,y) \| \sup_{(t,x)} \| w(t,x) \|.$$ 
Choose an orthonormal basis $\{e_i\}_{i\in I}$ in $\HH$ and let 
$v_i(s,y) := \ip{v(s,y)}{e_i}$ and $w_i(t,x) := \ip{e_i}{w(t,x)}$.
Then
\[
    \psi(ts^{-1}, xy^{-1}) = \sum_{i\in I} v_i(s,y) w_i(t,x), \ \ \ \ \ \ s,t\in G, x,y\in \Gamma.
\]
Let $S$ be the completely bounded operator on $\Bd[L^2(G\times\Gamma)]$,
given by $S(T) := \sum_{i \in I} M_{w_i} T M_{v_i}$.
Clearly,
\begin{equation}\label{eq_BGGa}
    \| S \|_\cb = \| \psi \|_{B(G \times \Gamma)}.
\end{equation}
To complete the proof, it suffices 
to show that the restriction of the operator $S$ to $C_r^*(\Gamma)\rtimes_{\beta,\id}^{w^*} G$ is given by
\begin{equation}\label{eq_Spi}
    S \big( \pi^\id(\lambda_x^\Gamma)\lambda^\id_s \big) = \pi^\id \big( \psi(s,x)\lambda_x^\Gamma \big)\lambda^\id_s.
\end{equation}
First note that 
\begin{equation}\label{eq_piid}
(\pi^\id (\lambda_x^\Gamma)\xi)(t) = \beta_{t^{-1}}(\lambda^\Gamma_x)\xi(t) = \overline{t(x)}\lambda^\Gamma_x\xi(t), 
\ \ \xi\in L^2(G,L^2(\Gamma)).
\end{equation}
Writing $v_i(t)(\cdot)$ and $w_i(t)(\cdot)$ for $v_i(t,\cdot)$ and $w_i(t,\cdot)$, respectively, 
for $t\in G$ and $y\in \Gamma$, and fixing 
$\xi,\eta \in L^2(G,L^2(\Gamma))$, 
we have
\begin{eqnarray*}
& &  
\left\langle S(\pi^\id (\lambda_x^\Gamma)\lambda^\id_s)\xi,\eta\right\rangle \\
& = & 
\sum_{i \in I} \left\langle M_{w_i}\pi^\id (\lambda_x^\Gamma)\lambda^\id_s M_{v_i}\xi,\eta \right\rangle \\
& = & 
\sum_{i \in I} \int \left( M_{w_i(t)} \overline{t(x)} \lambda_x^\Gamma M_{v_i(s^{-1}t)}\xi(s^{-1}t) \right)(y) \overline{\eta(t,y)} dtdy \\
& = & 
\sum_{i \in I} \int w_i(t,y) v_i(s^{-1}t,x^{-1}y) \overline{t(x)}\xi(s^{-1}t, x^{-1}y) \overline{\eta(t,y)} dtdy \\
& = & 
\int \psi(tt^{-1}s, yy^{-1}x)\overline{t(x)}\xi(s^{-1}t, x^{-1}y)\overline{\eta(t,y)} dtdy \\
& = & 
\int \psi(s,x) \big( \pi^\id(\lambda_x^\Gamma)\lambda^\id_s\xi \big)(t,y) \overline{\eta(t,y)} dtdy.
\end{eqnarray*}
Together with (\ref{eq_piid}), this establishes (\ref{eq_Spi}). 
In addition, 
\[
    \| \psi \|_\Sch = \left\| S \big|_{C^*_r(\Gamma) \rtimes_{\beta,\id}^{w^*} G} \right\|_\cb \leq \| S \|_\cb = \| \psi \|_{B(G\times \Gamma)} ,
\]
which together with (\ref{eq_BGGa}) gives the desired equality.

To see that the identification is multiplicative, observe that if 
$\psi , \chi \in \HSconv^\id(G)$ then $S_{F_\psi} S_{F_\chi} = S_{F_{\psi\chi}}$.
\end{proof}

In Theorem~\ref{th:convmultsnonab} below we will show that the
identification in Theorem \ref{conv_ab} is in fact a complete isometry.

\subsection{General case}
\label{ssec:convmultsgeneral}

Now let $G$ be an arbitrary locally compact group. 
In order to define convolution multipliers, we replace $C^*_r(\Gamma)$ with the quantum group dual of 
$C^*_r(G)$, namely $C_0(G)$, equipped with its natural action of $G$.
Similarly we replace $B(\Gamma)$ by $M(G)$, the Banach algebra of all complex-valued Radon measures on $G$ 
with the convolution multiplication, given by
\[
    (\mu \ast \nu)(f) := \int_G \int_G f(st) d\mu(s)\, d\nu(t) , \quad f\in C_0(G),\ \mu , \nu\in M(G) .
\]
We identify $L^1(G)$ with the norm-closed ideal in $M(G)$ consisting of absolutely continuous measures 
with respect to left Haar measure. We have that $L^1(G)$ is an $M(G)$-bimodule in the natural way. 
Using the identification $L^1(G)^* = L^\infty(G)$, we arrive at an $M(G)$-bimodule structure on 
$L^\infty(G)$, given by
\[
    \dualp{\mu\cdot f}{h} = \dualp{f}{h \ast \mu} \text{ and } \dualp{f \cdot \mu}{h} = \dualp{f}{\mu \ast  h} ,
\]
for $h\in L^1(G)$, $f\in L^\infty(G)$, $\mu\in M(G)$.
In particular,
\[
    (\mu\cdot f)(s) = \int_G f(st)d\mu(t) \quad  \text{ and } \quad (f\cdot \mu)(t) = \int_G f(st)d\mu(s).
\]
Let $\rho$ be the right regular representation of $G$ on $L^2(G)$; thus,
\[
    (\rho_s\xi)(t) = \Delta(s)^{1/2}\xi(ts).
\]
For $\mu\in M(G)$, define a bounded linear operator $\theta(\mu)(a)$, $a \in \Bd[L^2(G)]$, by 
\[
    \theta(\mu)(a) := \int_G \rho_sa\rho_s^*d\mu(s) .
\]
By \cite[Theorem 3.2]{NRS08} (see also \cite[Theorem 4.5]{Neu00}), the map $\theta$ is a 
weak$^*$-weak$^*$ continuous completely isometric homomorphism from $M(G)$ to the space $\cbw(\Bd[L^2(G)])$ 
of all completely bounded weak* continuous linear maps on $\mathcal{B}(L^2(G))$ and
$\|\theta(\mu)\|_\cb = \|\theta(\mu)\| = \|\mu\|$. 
We have
\[
    \theta(\mu)(f) = \mu\cdot f\in L^\infty(G), \quad f\in L^\infty(G). 
\]
Moreover, $\theta(\mu)$ is a $\vN(G)$-bimodule map.

For each $t \in G$, let $\beta_t : L^\infty(G) \to L^\infty(G)$ be given by $\beta_t(f) := \lambda^G_t f \lambda^G_{t\inv} = f_t$, where $f_t(x)=f(t^{-1}x)$.
Then 
\begin{equation}\label{eq_betcom}
    \beta_t \circ \theta(\mu) = \theta(\mu) \circ \beta_t, \quad t\in G.
\end{equation}

For $\Lambda = \{\mu_t\}_{t\in G}\subseteq M(G)$, define $F_\Lambda : G \to \cbm(C_0(G))$ by
\[
	F_\Lambda(t)(f) := \theta(\mu_t)(f) , \quad t \in G,\ f \in C_0(G) .
\]

\begin{definition}\label{de:convmultgeneral}
A family $\Lambda = \{\mu_t\}_{t\in G}\subseteq M(G)$ is called a 
\emph{convolution multiplier} if $F_\Lambda$ is a Herz--Schur $(C_0(G),G,\beta)$-multiplier.
\end{definition}

\noindent 
If $\Lambda = \{\mu_t\}_{t\in G}$ is a convolution multiplier, we set $\|\Lambda\|_\HS = \|F_\Lambda\|_\HS$.

Let $\id$ denote the representation of $C_0(G)$ on $L^2(G)$ by multiplication operators and 
$\HSconv^\id(G)$ be the collection of families $\Lambda = \{\mu_t\}_{t\in G}\subseteq M(G)$ such that $F_\Lambda$ is a Herz--Schur $\id$-multiplier of $(C_0(G),G,\beta)$, endowed with the algebra structure coming from pointwise operations on the maps $F_\Lambda$.
When $G$ is abelian, 
the identifications $C_0(G) \equiv C^*_r(\Gamma)$ and $M(G) \equiv B(\Gamma)$ 
show that the usage of the notation $\HSconv^\id(G)$ agrees with that from Subsection \ref{ssec:abelianconvmults}.

Consider the operator space projective tensor product $L^1(G) \projtens A(G) = (L^\infty(G) \btens \vN(G))_*$. 
We note that, when equipped with the product, given on elementary tensors by 

\[
    (f\otimes u) (g\otimes v) = (f\ast g)\otimes (uv), \quad f,g\in L^1(G),\ u,v\in A(G),
\]
the operator space $L^1(G) \projtens A(G)$ is a completely contractive Banach algebra. 
A map $T \in \Bd[L^1(G) \projtens A(G)]$ will be called a \emph{right multiplier of $L^1(G) \projtens A(G)$} if 
\[
    T(ab) = a T(b), \quad a,b \in L^1(G) \projtens A(G).
\]
If, in addition, $T$ is completely bounded, 
we write $T\in \Mcb^r(L^1(G) \projtens A(G))$, and 
call $T$ a \emph{right completely bounded multiplier of $L^1(G) \projtens A(G)$}.
When $G$ is abelian we have the identifications 
\[
    \Mcb^r(L^1(G) \projtens A(G)) = \Mcb(A(\Gamma \times G)) = B(\Gamma \times G).
\]
Our goal is to generalise Theorem~\ref{conv_ab}, 
identifying $\HSconv^\id(G)$ with the space of right completely bounded multipliers $\Mcb^r(L^1(G) \projtens A(G))$.

If $M$ is any of the von Neumann algebras $L^{\infty}(G)$, $\vN(G)$ or $L^\infty(G) \btens \vN(G)$, 
$T\in M$ and $f\in M_*$, we write $f\cdot T$ and $T\cdot f\in M$ for the operators given by 
\[
	\dualp{f \cdot T}{g} \defeq \dualp{T}{gf}, \quad  \dualp{T\cdot f}{g} \defeq \dualp{T}{fg}, 
	\quad g\in M_*,
\]
where $\dualp{\cdot}{\cdot}$ is the pairing between $M$ and $M_*$. 
We recall \cite{eymard} 
that the support of $T\in \vN(G)$ is the closed set of all $t\in G$ such that $u\cdot T\ne 0$ whenever $u\in A(G)$ and $u(t)\ne 0$.

\begin{lemma}\label{le:rightmultmeasures}
If $T\in \Mcb^r(L^1(G) \projtens A(G))$ then there exists a unique family $\{\mu_t\}_{t\in G}\subseteq M(G)$ such that 
\[
    T^*(f\otimes\lambda^G_t) = \theta(\mu_t)(f)\otimes\lambda^G_t, \quad f\in L^{\infty}(G), t\in G.
\]
\end{lemma}
\begin{proof}
Let $f_1 , f_2\in L^1(G)$, $a_1 , a_2\in A(G)$. The equality
\[
    T((f_1\otimes a_1)(f_2\otimes a_2)) = (f_1\otimes a_1)T(f_2\otimes a_2)
\]
implies that, if $g\in L^\infty(G)$ then
\begin{equation}\label{pairing}
    \dualp{T^*(g\otimes\lambda^G_t)}{(f_1\otimes a_1)(f_2\otimes a_2)} = a_1(t) \dualp{T^*(g\cdot f_1 \otimes \lambda^G_t)}{f_2\otimes a_2}.
\end{equation}
Taking the limit along an approximate identity $\{f_\alpha\}_{\alpha\in \mathbb{A}}$ of $L^1(G)$, we obtain 
\begin{equation}\label{eq_Tst}
    \langle T^*(g\otimes\lambda^G_t),f_2\otimes a_1a_2\rangle = \langle a_1(t) T^*(g \otimes\lambda^G_t), f_2\otimes a_2\rangle.
\end{equation}
For $\omega\in L^1(G)$, let 
$R_\omega : L^\infty(G) \btens \vN(G) \to \vN(G)$ be the slice map, 
defined by 
$$\langle R_\omega(S),a\rangle := \langle S, \omega\otimes a\rangle, \ \ S\in L^\infty(G) \btens \vN(G), a\in A(G).$$
After taking a limit along an approximate unit for $L^1(G)$, equation (\ref{eq_Tst}) implies that
\[
    a_1\cdot R_\omega \big( T^*(g\otimes\lambda^G_t) \big) = a_1(t) R_\omega \big( T^*(g\otimes\lambda^G_t) \big).
\]
It follows that $R_\omega(T^*(g\otimes\lambda^G_t))\in \vN(G)$ has support in $\{ t \}$. 
By \cite[Th\'eor\`eme 4.9]{eymard}, 
$R_\omega(T^*(g\otimes\lambda^G_t)) = c(\omega,t)\lambda^G_t$ for some constant $c(\omega,t)$ and 
\[
    R_\omega(T^*(g\otimes\lambda^G_t)(1\otimes\lambda^G_{t^{-1}}))\in \CC I.
\]
By \cite{kraus},
$T^*(g\otimes\lambda^G_t)(1\otimes\lambda^G_{t^{-1}}) \in L^\infty(G) \otimes \mathbb{C}I$ and hence
$T^*(g\otimes\lambda^G_t) = g_t\otimes\lambda^G_t$ for some $g_t\in L^\infty(G)$.

The map $\Phi_t : g \mapsto g_t$ is completely bounded, normal, and 
$T^*(g\otimes\lambda^G_t) = \Phi_t(g)\otimes\lambda^G_t$, $t\in G$. 
By (\ref{pairing}),
\[
    \Phi_t(g\cdot f_1)=\Phi_t(g)\cdot f_1, \quad f_1\in L^1(G) ,
\]
showing that $(\Phi_t)_*(f_1\ast f_2) = f_1 \ast ((\Phi_t)_*(f_2))$. 
Thus $(\Phi_t)_*$ is a right completely bounded multiplier of $L^1(G)$. 
By \cite[Theorem 3.2]{NRS08} (see also \cite[Theorem 4.5]{Neu00}), there exist $\{\mu_t\}_{t\in G}$ such that $\Phi_t(g) = \theta(\mu_t)(g)$. 
\end{proof}

In what follows we will speak of a family $\Lambda = \{\mu_t\}_{t\in G} \subseteq M(G)$
being a convolution multiplier or a (completely bounded) right multiplier. 
For a right multiplier $\Lambda$ of $L^1(G) \projtens A(G)$ we denote by $R_{\Lambda}$ the mapping on $L^\infty(G) \btens \vN(G)$, given by
\begin{equation}\label{eq:convMcbr}
    R_\Lambda (f \otimes \lambda^G_r) := \theta(\mu_r)(f) \otimes \lambda^G_r .
\end{equation}

\begin{theorem}\label{th:convmultsnonab}
Let $\Lambda=\{\mu_t\}_{t\in G} \subseteq M(G)$. 
The following are equivalent:
\begin{enumerate}[i.]
    \item $\Lambda \in \HSconv^\id(G)$;
    \item $\Lambda\in \Mcb^r(L^1(G) \projtens A(G))$.
\end{enumerate}
The identification $R_\Lambda \mapsto S_{F_\Lambda}$ is a completely isometric algebra isomorphism.
\end{theorem}
\begin{proof}
(i)$\implies$(ii) 
We identify $C_0(G)\rtimes_{\beta ,\id}^{w^*} G$ with the von~Neumann algebra crossed product 
$L^\infty(G)\rtimes_\beta^{\rm vN} G$, and
let $\Lambda = \{\mu_t\}_{t\in G}$ be a convolution multiplier. 
For $f\in L^\infty(G)$, using (\ref{eq_betcom}) we have
\[
\begin{split}
    \cl N(F_\Lambda)(s,t)(f) &= \beta_{t^{-1}} \big( F_\Lambda(ts^{-1})(\beta_t(f)) \big) \\ 
        &= \beta_{t^{-1}}\big( \theta(\mu_{ts^{-1}})(\beta_t(f)) \big) = \theta(\mu_{ts^{-1}})(f).
\end{split}
\]
Following similar arguments as in the proof of \cite[Theorem 3.8]{mtt}, we obtain that there exist a normal $*$-representation $\rho$ of $L^\infty(G)$ on $\HH_\rho$ and $\mathcal{V}, \mathcal{W} \in L^{\infty}(G, \Bd[L^2(G),\HH_\rho])$ such that
\[
    \theta(\mu_{ts^{-1}})(f) = \mathcal{W}^*(t) \rho(f) \mathcal{V}(s)
\]
and $\|\Lambda\|_\Sch = \esssup_{s\in G} \| \mathcal{V}(s) \| \esssup_{t\in G} \|\mathcal{W}(t) \|$. 

Define a map $R_\Lambda: L^\infty(G) \btens \vN(G)\to \Bd[L^2(G)\otimes L^2(G)]$ by
\[
    R_\Lambda(f\otimes\lambda^G_t) := W^*(\rho(f)\otimes\lambda^G_t)V,
\]
where $V,W\in \Bd[L^2(G, \HH_\rho\otimes L^2(G))]$ are given by $(V\xi)(t) = \mathcal{V}(t)\xi(t)$, $(W\xi)(t) = \mathcal{W}(t)\xi(t)$.
Then
\[
\begin{split}
    R_\Lambda(f \otimes \lambda^G_t)\xi(s) = \mathcal{W}^*(s) \rho(f) \mathcal{V}(t^{-1}s)\xi(t^{-1}s) &= \theta(\mu_{s(s^{-1}t)})(f)\xi(t^{-1}s) \\
        &= (\theta(\mu_t)(f) \otimes \lambda_t^G \xi)(s) .
\end{split}
\]
In particular, $R_\Lambda (f\otimes\lambda^G_t) \in L^\infty(G) \btens \vN(G)$, and hence $R_\Lambda$ is a normal completely bounded map on $L^\infty(G) \btens \vN(G)$.
Moreover, if $f_1 , f_2\in L^1(G)$, $a_1 , a_2\in A(G)$, $g\in L^\infty(G)$, and $(R_\Lambda)_*$ is the predual
of $R_{\Lambda}$, we have 
\[
\begin{split}
    &\dualp{g\otimes\lambda^G_t}{(R_\Lambda)_*((f_1\otimes a_1)(f_2\otimes a_2))} = \dualp{\theta(\mu_t)(g)\otimes\lambda^G_t}{f_1\ast f_2\otimes a_1a_2} \\
        &\quad = \dualp{\mu_t \cdot g}{f_1\ast f_2} \dualp{\lambda^G_t}{a_1a_2} = \dualp{g}{(f_1 \ast f_2) \ast \mu_t} \dualp{a_1\cdot  \lambda^G_t}{a_2} \\
        &\quad = \dualp{g\cdot f_1}{f_2 \ast \mu_t} \dualp{a_1 \cdot \lambda^G_t}{a_2} = \dualp{\mu_t \cdot (g \cdot f_1)}{f_2} \dualp{a_1(t) \lambda^G_t}{a_2} \\
        &\quad = \dualp{R_\Lambda(g \cdot f_1 \otimes a_1(t)\lambda^G_t)}{f_2\otimes a_2} = \dualp{g \cdot f_1\otimes a_1(t) \lambda^G_t}{(R_\Lambda)_*(f_2\otimes a_2)} \\
        &\quad = \dualp{g\otimes \lambda^G_t}{(f_1\otimes a_1) (R_\Lambda)_*(f_2\otimes a_2)} ,
\end{split}
\]
\textit{i.e.}
\[
    (R_\Lambda)_*((f_1\otimes a_1)(f_2\otimes a_2)) = (f_1\otimes a_1) (R_\Lambda)_*(f_2\otimes a_2).
\]
Hence $(R_\Lambda)_*(ab) = a (R_\Lambda)_*(b)$ for any $a ,b \in L^1(G) \projtens A(G)$ and therefore $(R_\Lambda)_*$ 
is a right completely bounded multiplier of $L^1(G)\widehat\otimes A(G)$.
In addition, 
\begin{equation}\label{eq_inlno2}
    \| R_\Lambda \|_\cb \leq \esssup_{s\in G} \| \cl{V}(s) \| \esssup_{t\in G} \| \cl{W}(t) \| = \| \Lambda \|_\Sch.
\end{equation}

(ii)$\implies$(i) 
Assume now that $\Lambda = \{\mu_t\}_{t\in G} \in \Mcb^r(L^1(G) \projtens A(G))$, \textit{i.e.}\ the map $f\otimes\lambda^G_t \mapsto \theta(\mu_t)(f) \otimes \lambda^G_t$ extends to a normal 
right $L^1(G) \projtens A(G)$-modular completely bounded map $R_\Lambda$ on $L^\infty(G) \btens \vN(G)$. 
By \cite[Proposition 4.3]{jnr}, there exists a unique 
$\vN(G)\btens L^\infty(G)$-bimodule map
$\widetilde{R_\Lambda} \in \cbw(\Bd[L^2(G \times G)])$ such that 
$\widetilde{R_\Lambda}|_{L^\infty(G) \btens \vN(G)} = R_\Lambda$
and $\|\widetilde{R_\Lambda}\|_{\cb}=\|R_\Lambda\|_{\cb}$.
We have, in particular, 
\begin{equation}\label{eq_modtw}
    \widetilde{R_\Lambda}(g\otimes f\lambda^G_t ) = \theta(\mu_t)(g) \otimes f\lambda^G_t , \quad f,g \in L^\infty(G).
\end{equation}

Note that $L^2(G\times G) \equiv L^2(G,L^2(G))$ and 
let $\pi : L^{\infty}(G) \to \cl B(L^2(G\times G))$ be the *-representation, given by 
$$\pi(f) \xi(t) = \beta_{t^{-1}}(f)(\xi(t)), \ \ \xi\in L^2(G\times G), f\in L^\infty(G).$$
Let $f\in L^{\infty}(G)$ and note that $\pi(f)\in L^{\infty}(G\times G)$. 
Thus, there exists a net 
$\{\omega_\alpha\}_{\alpha\in \mathbb{A}} \subseteq \lspan\{ g\otimes h : g,h \in L^\infty(G) \}$, 
with 
$\omega_{\alpha}\to_{\alpha\in \mathbb{A}} \pi(f)$ in the weak* topology.  
Write $\omega_\alpha = \sum_{i = 1}^{n_{\alpha}} g_{i,\alpha} \otimes h_{i,\alpha}$.  
Using (\ref{eq_betcom}) and (\ref{eq_modtw}), we have 

\[
    \big( \widetilde{R_\Lambda} (\pi(f)(1 \otimes\lambda^G_r) \big) 
    = \lim_{\alpha\in \mathbb{A}} \sum_{i =1}^{n_{\alpha}} \theta(\mu_r)(g_{i,\alpha}) \otimes h_{i,\alpha} \lambda^G_r 
    = \pi \big( \theta(\mu_r)(f) \big) (1\otimes\lambda^G_r).
\]
Since 
$$(\widetilde{S_{F_{\Lambda}}} (\pi(f)(1\otimes\lambda^G_t)) = (\pi(\theta(\mu_t)(f))(1\otimes\lambda^G_t),$$ 
the restriction of $\widetilde{R_\Lambda}$ to the crossed product $C_0(G)\rtimes_{\beta,r}G$ coincides with $S_{F_\Lambda}$, 
implying the converse statement. 
Note, in addition, that 
\begin{equation}\label{eq_inlno}
\| S_{F_\Lambda} \|_\cb \leq \| \widetilde{R_\Lambda} \|_\cb = \| R_\Lambda\|_\cb.
\end{equation}

By (\ref{eq_inlno2}), 
$\| R_\Lambda \|_\cb \leq \| S_{F_\Lambda} \|_\cb$,
and together with (\ref{eq_inlno}) this shows that $\| R_\Lambda \|_\cb = \| S_{F_\Lambda} \|_\cb$.
Moreover, by Lemma~\ref{le:Ncompleteisomhom} the map $F_\Lambda \mapsto \mathcal{N}(F_\Lambda)$ is a complete isometry, and by \cite[Proposition 4.3]{jnr} the map $R_\Lambda \mapsto \widetilde{R_\Lambda}$ is a complete isometry, 
therefore the norm 
inequalities hold on all matrix levels, implying that the identification $S_{F_\Lambda} \mapsto R_\Lambda$ is a complete isometry.

The homomorphism claim follows from Lemma~\ref{le:Ncompleteisomhom} and the fact that the identification in \cite[Proposition 4.3]{jnr} is a homomorphism.
\end{proof}

We observe that the product of the convolution multipliers $\Lambda = \{ \mu_t \}_{t \in G}$ and 
$\Xi = \{ \nu_t \}_{t \in G}$ is given by $\Lambda \Xi = \{ \mu_t \ast \nu_t \}_{t \in G}$.
We write $\HScent(A , G ,\alpha)$ for the central Herz--Schur $(A,G,\alpha)$-multipliers.

\begin{proposition}\label{le:intersecconvcentr}
We have $\HSconv(G) \cap \HScent(C_0(G) , G ,\beta) = \Mcb A(G)$. 
\end{proposition}
\begin{proof}
Suppose that $F : G \to \cbm(C_0(G))$ is a central multiplier which is also a convolution multiplier. Then for each $r \in G$ there is $a_r \in C_b(G)$ such that $F(r)(a) = a_r a$.
Also, since $F$ is a convolution multiplier, by (\ref{eq_betcom}) $F(r)$ satisfies 
\[
    \beta_t \big( F(r)(a) \big) = F(r) \big( \beta_t(a) \big) , \quad r,t \in G ,\ a \in C_0(G) .
\]
Combining these two identities, and allowing $a$ to vary, gives $a_r(st) = a_r(t)$ for all $s,t \in G$, so $a_r$ is a scalar multiple of the identity.
The conclusion follows from \cite[Proposition 4.1]{mtt}.
\end{proof}

\section{Idempotent multipliers}
\label{sec:idemp}

Given standard measure spaces $(X,\mu)$ and $(Y,\nu)$, a well-known open problem
asks for the identification of the idempotent Schur multipliers on $X \times Y$.
A characterisation of the \emph{contractive} idempotent Schur multipliers, based on a combinatorial argument,
combined with an observation of Livshitz \cite{Liv95}, was given by 
Katavolos--Paulsen in \cite{KP05}.

In a similar vein, 
for a general locally compact group $G$, 
there is no known characterisation of the idempotent Herz--Schur multipliers.
Some partial results are known:
the idempotent measures in $M(G)$ of norm one were characterised by 
Greenleaf~\cite{Gr65}
--- a measure $\mu$ has  the properties $\mu\ast\mu = \mu$  and $\|\mu\| = 1$ if and only if $\mu = \gamma m_H$, where $m_H$ is  the Haar measure on a  compact subgroup $H$ and $\gamma$ is a character of $H$. Such $\mu$ is positive if and only if $\gamma$ above is equal to $1$.
Dually, the idempotent elements of $B(G)$ were characterised by Host~\cite{Hos86}; 
using Host's method, Ilie and Spronk \cite{IlSp05} characterised contractive idempotents 
--- a function $u \in B(G)$ has the properties $u^2 = u$ and $\| u \| = 1$ if and only if $u = \chi_C$, where $C$ is an open coset of $G$.
Such $u$ is positive if and only if $C$ is a subgroup of $G$. Stan~\cite{Stan09} extended this characterisation to norm one
idempotent elements of $\Mcb A(G)$. 

In this section we use the aforementioned 
results of Katavolos--Paulsen and Stan to 
study the idempotent central and the idempotent convolution multipliers.

\subsection{Central idempotent multipliers}\label{ssec:idempcent}

We fix standard measure spaces $(X,\mu)$ and $(Y,\nu)$ and a separable, non-degenerate $C^*$-algebra
$A\subseteq \mathcal{B(H)}$.
Suppose $\varphi \in L^{\infty}(X\times Y)$ is an idempotent Schur multiplier, so the map 
$k \mapsto \varphi \cdot k$ on $L^2(Y\times X)$ gives rise to 
a bounded idempotent map $S_\varphi$ on the space of compact operators; 
we have that $\varphi^2(x,y)k(y,x) = \varphi(x,y) k(y,x)$ almost everywhere for all $k\in L^2(Y\times X)$, which implies 
that $\varphi^2 = \varphi$. 
By \cite[Proposition 11]{KP05}, $\varphi = \chi_E$ almost everywhere for some $\omega$-open and $\omega$-closed $E\subseteq X\times Y$. 

Recall from \cite{KP05} that a subset 
$E\subseteq X\times Y$ is said to have the \emph{3-of-4 property} provided that given any distinct pair of points $x_1\neq x_2$ in $X$ and any pair of distinct pairs $y_1\neq y_2$ in $Y$, whenever 3 of the 4 ordered pairs $(x_i,y_j)$ belong to $E$ then 
the fourth one also belongs to $E$.

For a subset $W\subseteq C\times Z$, where $C$ is a set (which will below be equal to either $X$ or $Y$), and an element $z\in Z$, we write 
$W_z = \{t\in C : (t,z)\in W\}$. 
The following result generalises \cite[Theorem 10]{KP05}.

\begin{proposition}\label{pr:idempcentralcharBorels} 
Let $(X,\mu)$ and $(Y,\nu)$ be standard measure spaces and $Z$ a locally compact Hausdorff space.
Let $\varphi:X\times Y\times Z\to \CC$ be a measurable function, continuous in the $Z$-variable. 
The following are equivalent:
\begin{enumerate}[i.]
    \item $\varphi$ is a contractive idempotent central Schur $C_0(Z)$-multiplier;
    \item for each $z \in Z$, there exist families
    $(A_i^z)_{i\in \mathbb{N}}$ and $(B_i^z)_{i\in \mathbb{N}}$ of pairwise disjoint measurable subsets of $X$ and $Y$, respectively, such that     
    $\varphi(x,y,z) = \sum_{i=1}^\infty \chi_{A_i^z}(x) \chi_{B_i^z}(y)$ almost everywhere.
\end{enumerate}
\end{proposition}

\begin{proof}
(i)$\implies$(ii) 
By Theorem~\ref{th_dia}, $\varphi_z$ is a contractive idempotent Schur multiplier for every $z\in Z$.
By \cite[Theorem 10]{KP05}, there exist
families $(A_i^z)_{i=1}^{\infty}$ and $(B_i^z)_{i=1}^{\infty}$
of pairwise disjoint measurable subsets of $X$ and $Y$, respectively, such that
$\varphi_z(x,y) = \sum_{i = 1}^\infty \chi_{A_i^z}(x) \chi_{B_i^z}(y)$ almost everywhere.

(ii)$\implies$(i) 
By \cite[Theorem 10]{KP05},  
$\varphi_z$ is a contractive idempotent Schur multiplier for every $z\in Z$;
thus, by Theorem~\ref{th_dia}, $\varphi$ is a central $C_0(Z)$-multiplier, which is easily seen to be idempotent. 
Since each $\varphi_z$ is contractive we have $\varphi$ is contractive by Theorem~\ref{th_dia}.
\end{proof}

\begin{remark}\label{discrete}\rm 
The statement holds when the standard measure spaces are replaced by discrete spaces $X$ and $Y$ with counting measures, but in this case the families $(A_i^z)_i$, $(B_i^z)_i$ might be uncountable if $X$ or $Y$ is uncountable.  In this case (i) is also equivalent to $\varphi=\chi_W$, where $W_z$ has the 3-of-4 property for each $z\in Z$, see \cite[Lemma 2]{KP05}.
\end{remark}

Let $Z$ be a locally compact Hausdorff space equipped with an action $\alpha$ of a 
locally compact group $G$.
In the subsequent results, we view the set $Z \times G$ as a groupoid as in Section~\ref{ssec:connectionscentral}.
We provide a combinatorial characterisation of 
the contractive central Herz--Schur $(C_0(Z), G,\alpha)$-multipliers.
It is easy to see that in this case $\psi(x,t) = \chi_V(x,t)$ for some subset $V \subseteq Z\times G$.
Theorem \ref{pr:groupoididempHScentralmult} generalises the result of Stan \cite[Theorem 3.3]{Stan09}.

\begin{theorem}\label{pr:groupoididempHScentralmult}
 Assume that
$V \subseteq Z\times G$  is a subset that is both closed and open. 
The following are equivalent:
\begin{enumerate}[i.]
    \item $F_{\chi_V}$ is a contractive central Herz--Schur $(C_0(Z), G,\alpha)$-multiplier;
    \item if $(x,t)$, $(x,s)$, $(xr, r^{-1}s)\in V$ then $(xr, r^{-1}t)\in V$; equivalently, if $(x,t)$, $(y,s)$, $(z,p)\in V$ and the product $(z,p)(y,s)^{-1}(x,t)$ is well defined then $(z,p)(y,s)^{-1}(x,t)\in V$.
\end{enumerate}
In particular, if $V = Z\times A$ for some $ A\subseteq G$ then $A$ is an open coset of $G$.
\end{theorem}
\begin{proof}
Let 
\[
	W = \{(x,s,t)\in Z\times G\times G: (xt^{-1}, ts^{-1})\in V \}.
\]
By Corollary \ref{co:centralmultscharMTT},
$F_{\chi_V}$ is a Herz--Schur $(C_0(Z), G,\alpha)$-multiplier if and only if 
the map $\cl N(F_{\chi_V})$, given by
$$\cl N(F_{\chi_V})(s,t)(a)(x)=\chi_V(xt^{-1}, ts^{-1})a(x)=\chi_W(x,s,t)a(x),$$ 
is a Schur $C_0(Z)$-multiplier.

We first show that condition (ii) is equivalent to $W_z := \{ (s,t) \in G \times G : (z,s,t) \in W \}$ having the 3-of-4 property for all $z \in Z$.
Suppose that $(z,t_1,s_1)$, $(z,t_1,s_2)$ and $(z,t_2,s_2)\in W$, which is equivalent to $(zt_1^{-1}, t_1s_1^{-1})$, $(zt_1^{-1}, t_1s_2^{-1})$, $(zt_2^{-1}, t_2s_2^{-1})\in V$.
Writing $zt_1^{-1} = x$, $t_1s_1^{-1} = t$, $t_1s_2^{-1} = s$ and $t_1t_2^{-1} = r$, 
we get $zt_2^{-1}=xr$,  $t_2s_1^{-1}=r^{-1}t$ and $t_2s_2^{-1}=r^{-1}s$  and hence $(x,t)$, $(x,s)$, $(xr, r^{-1}s)\in V$.
The condition $(z,t_2,s_1)\in W$ is equivalent to $(xr, r^{-1}t)\in V$, giving the statement.
We note that $(z,p)(y,s)^{-1}(x,t)=(z,p)(ys,s^{-1})(x,t)$ is well defined if and only if 
$y = x$ and $z = xsp^{-1}$;
letting $r = sp^{-1}$, we have $(z,p) = (xr,r^{-1}s)$.
We have shown that condition (ii) is equivalent to the 3-of-4 property for each $W_z$.

Assume first that $G$ is a locally compact second countable group and hence $(G,m_G)$ is a standard measure space. 

(i)$\implies$(ii) If (i) holds then $\cl N(F_{\chi_V})$ is a 
contractive idempotent Schur $C_0(Z)$-multiplier. 
By Theorem~\ref{th_dia}, 
$\varphi_z = \chi_{W_z}$ is a contractive idempotent Schur multiplier for each $z\in Z$.
By \cite[Theorem 10]{KP05}, there exist 
countable collections $\{I_m\}$ and $\{J_m\}$ of mutually disjoint Borel subsets of $G$, such that,
if $E=\cup_mI_m\times J_m$, then $\chi_{W_z} = \chi_E$ almost everywhere.
 
As $\chi_{W_z}$ is continuous and hence $\omega$-continuous and $\chi_E$ is $\omega$-continuous, by \cite[Lemma 2.2]{STT11}, $\chi_{W_z}=\chi_E$ marginally almost everywhere. Hence there exists a null set $N_z$ such that $\chi_{W_z}=\chi_E$ on $N_z^c\times N_z^c$. In particular, $W_z\cap(N_z^c\times N_z^c)$ has the  3-of-4 property. 
To see that the whole $W_z$ has the property, take $t_1$, $t_2$, $s_1$, $s_2$ such that $(t_1, s_1)$, $(t_1, s_2)$, $(t_2,s_2)\in W_z$, but some of $t_1$, $s_1$, $t_2$, $s_2$ belong to $N_z$. Using the fact that $W_z$ is open and $m(N_z)=0$ we can find sequences $(t_1^n)_n$,  $(s_1^n)_n$, $(t_2^n)_n$, $(s_2^n)_n$ of elements in $N_z^c$ such that 
$(t_1^n, s_1^n)$, $(t_1^n, s_2^n)$, $(t_2^n,s_2^n)\in W_z$ and $t_i^n\to t_i$, $s_i^n\to s_i$, $i=1,2$. Hence $(t_2^n,s_1^n)\in W_z$, and as $1=\chi_{W_z}(t_2^n,s_1^n)\to\chi_{W_z}(t_2,s_1)$, we obtain that $(t_2,s_1)\in W_z$.  Hence (ii) holds.

(ii)$\implies$(i)
 As $W_z$ is open and hence $\omega$-open, $W_z$ is marginally equivalent to a countable union of Borel rectangles. Hence $W_z\cap (N_z^c\times N_z^c)=\cup_{m=1}^\infty A_m^z\times B_m^z$, where $m_G(N_z)=0$ and each $A_m^z\times B_m^z$ is Borel. 
By \cite[Lemma 2]{KP05} and the second paragraph in the proof, $W_z$ and hence $W_z\cap (N_z^c\times N_z^c)$ has the 3-of-4 property
for each $z \in Z$ and there exist families $\{ X_i^z \}_{i\in I}$ and $\{ Y_i^z \}_{i \in I}$ of pairwise disjoint sets of $G$, such that 
$W_z\cap (N_z^c\times N_z^c) = \cup_{i \in I} X_i^z \times Y_i^z$.
Arguing as in the proof of \cite[Theorem 10]{KP05} one shows that the index set $I$ can be chosen countable and each $X_i^z \times Y_i^z$ is a Borel rectangle. Hence $\chi_{W_z}$ is a contractive Schur multiplier. 
By Proposition~\ref{pr:idempcentralcharBorels} $\chi_W$ is a contractive idempotent central Schur multiplier, so $\chi_V$ is a contractive idempotent central Herz--Schur $(C_0(Z), G,\alpha)$-multiplier. 

If $G$ is discrete, the statement follows from Remark \ref{discrete}. 
Finally, if $V= Z \times A$ then $\chi_V(x,t)=\chi_A(t)$ which is a Herz--Schur $(C_0(Z),G,\alpha)$-multiplier if and only if $\chi_A$ is a Herz--Schur multiplier. It is of norm at most 1 if and only if $A$ is an open coset of $G$.
\end{proof}

\begin{remark}\label{re:3of4transfgroupoid}
{\rm
It follows from Proposition~\ref{pr:groupoididempHScentralmult} that if $F_{\chi_V}$ is a contractive Herz--Schur $(C_0(Z),G,\alpha)$-multiplier and the points $(x,t)$, $r((x,t))=(x,e)$ and $d((x,t))=(xt,e)$ all belong to $V$ 
then $(x,t)^{-1}=(xt,t^{-1})\in V$.
Moreover, if $(x,t)$, $d((x,t))=(xt,e)$ and $(xt,s)\in V$ then $(x,t)(xt,s)=(xt,ts)\in V$.
}
\end{remark}

The following corollary is an immediate consequence of Remark~\ref{re:3of4transfgroupoid}. 

\begin{corollary}
With the notation of Theorem \ref{pr:groupoididempHScentralmult}, 
assume that $\cl G_0\subseteq V$. We have that $F_{\chi_V}$ is a contractive 
Herz--Schur $(C_0(Z),G,\alpha)$-multiplier if and only if $V$ is a subgroupoid of $\cl G$. 
\end{corollary}

\subsection{Positive central idempotent multipliers}
\label{ssec:poscentridemps}

The following description of positive contractive Schur multipliers can be obtained in a similar manner to 
\cite[Theorem 10]{KP05}, and we omit its proof. 

\begin{proposition}\label{th:positive_contractive}
Let $(X,\mu)$ be a standard measure space and $E\subseteq X\times X$.  
The following are equivalent:
\begin{enumerate}[i.]
	\item $\chi_E$ is a positive contractive Schur multiplier;
	\item $E$ is equivalent, with respect to product measure, 
	to a subset of the form $\cup_{m=1}^{\infty} I_m\times I_m$, where $\{I_m\}_{m=1}^{\infty}$
	is a collection of disjoint Borel subset of $X$.
\end{enumerate}
\end{proposition}

\begin{remark}\rm
The standard measure space $(X,\mu)$ can be replaced by discrete space $X$ with counting measure. In this case the collection of disjoint subsets of $X$ might be uncountable. 
\end{remark}
The following positive version of Proposition \ref{pr:idempcentralcharBorels} and its discrete version
can be proved using similar ideas, and we omit the detailed argument.

\begin{proposition}\label{p_cpics} 
Let $(X,\mu)$ and $(Y,\nu)$ be standard measure spaces and $Z$ a locally compact Hausdorff space.
Let $\varphi:X\times Y\times Z\to \CC$ be a measurable function which is continuous in the $Z$-variable. 
The following are equivalent:
\begin{enumerate}[i.]
    \item $\varphi$ is a positive contractive idempotent central Schur $C_0(Z)$-multiplier;
    \item for each $z \in Z$, there exists a family
    $(A_i^z)_i$ of pairwise disjoint measurable subsets of $X$, such that     
    $\varphi(x,y,z) = \sum_{i=1}^\infty \chi_{A_i^z}(x) \chi_{A_i^z}(y)$ almost everywhere.
\end{enumerate}
\end{proposition}

Proposition \ref{pr:idempcentralcharBorels} and the transference theorem of \cite{mtt} give an 
implicit characterisation of the positive central idempotent Herz-Schur 
$(C_0(Z),G,\alpha)$-multipliers. In Theorem \ref{th_pchs} below, 
we give a more direct description of the positive central idempotent Herz-Schur multipliers
of norm not exceeding 1.

\begin{theorem}\label{th_pchs}
Let 
 $(C_0(Z), G, \alpha)$ be a $C^*$-dynamical system and 
$V\subseteq Z\times G$ be a closed and open subset. 
The following are equivalent:
\begin{enumerate}[i.]
	\item $F_{\chi_V}$ is a positive, contractive Herz--Schur $(C_0(Z), G, \alpha)$-multiplier;
	\item $V$ is a subgroupoid of $Z \times G$.
\end{enumerate}
\end{theorem}

\begin{proof}
We will prove the theorem for $G$ a locally compact second countable group. The case when $G$ is discrete can be treated in a similar but simpler way.

(i)$\implies$(ii) 
Let 
$$W = \{(z,s,t)\in Z\times G\times G: (zt^{-1}, ts^{-1})\in V\}.$$
If $F_{\chi_V}$ is a positive contractive Herz--Schur $(C_0(Z), G, \alpha)$-multiplier then 
the function $\cl N(F_{\chi_V})$, given by
$\cl N(F_{\chi_V})(s,t)(a)(z) = \chi_W(z,s,t)a(z)$, is a positive Schur $C_0(Z)$-multiplier.
By Theorem~\ref{th_poscentral}, $\chi_{W_z}$ is a positive Schur multiplier for each $z\in Z$.
Note also that, as it is continuous, it is $\omega$-continuous. 
Using \cite[Lemma 2.2]{STT11}, we see that 
there exist a weakly measurable function $v_z : G \to \ell^2$ and a null set 
$N_z\subseteq G$ such that
\[
	\chi_{W_z}(s,t) = \ip{v_z(s)}{v_z(t)}, \quad s,t\notin N_z.
\]
Let $(x,t)\in V$; as in Remark~\ref{re:3of4transfgroupoid}, 
it suffices 
to show that $(x,e)$ and $(xt,e)\in V$. Assume that $(x,e)\notin V$, and note that
\[
	\chi_V(x,e) = \chi_V((xt)t^{-1}, t t^{-1})=\chi_W(xt,t,t) \text{ and } \chi_V(x,t)=\chi_W(xt,e,t).
\]
If $t\notin N_{xt}$ and $e\notin N_{xt}$ then 
\[
	\chi_W(xt,t,t) = \| v_{xt}(t) \|_2^2 = 0 \text{ and } \chi_W(xt,e,t) = \ip{v_{xt}(e)}{v_{xt}(t)} =0,
\]
giving a contradiction.
If one or both of $e$ or $t$ are in $N_{xt}$, say $t\in N_{xt}$ but $e\notin N_{xt}$, then, as $m(N_{xt})=0$ there exists a sequence $s_n\notin N_{xt}$ such that $s_n\to t$. 
As $\chi_W$ is continuous, we obtain
\[
	\| v_{xt}(s_n) \|_2^2 = \chi_W(xt,s_n,s_n) \to \chi_W(xt,t,t) = 0 ,
\]
while 
\[
	\ip{v_{xt}(e)}{v_{xt}(s_n)} = \chi_W(xt,e,s_n) \to \chi_W(xt,e,t) , 
\]
forcing $\chi_W(xt,e,t)=0$, a contradiction. 
The other cases are treated similarly. 
To see that $(xt,e)\in V$ observe that
\[
	\chi_V(xt,e)=\chi_W(x,t^{-1},t^{-1}) \text{ and }\chi_V(x,t)=\chi_W(x,t^{-1},e)
\]
and apply similar analysis. 

(ii)$\implies$(i) Let now $V$ be an open subgroupoid. 
Arguing as in the proof of Proposition~\ref{pr:groupoididempHScentralmult} we see that 
$W_z$ has the 3-of-4 property for each $z\in Z$. 
Moreover, if $(x,s,t)\in W$ we have that $(xt^{-1}, ts^{-1})\in V$ and hence $r(xt^{-1}, ts^{-1})=(xt^{-1},e)\in V$ and $d(xt^{-1}, ts^{-1})=(xs^{-1},e)\in V$, implying $(x,t,t)\in W$ and $(x,s,s)\in W$. Therefore the projections $W_z^1$ and $W_z^2$ of $W_z$ 
on the first and the second coordinates are equal and $\{(s,s): s\in W_z^1\} \subseteq W_z$. 
It follows easily now that for each $z\in Z$ there exists disjoint sets $\{X_t^z\}_{t\in T}$ such that $W_z=\cup_{t\in T}X_t^z\times X_t^z$. Arguing as in \cite[Theorem 10]{KP05}, there is a Borel subset $N_z$, $m_G(N_z)=0$ such that $(X_t^z\cap N_z^c)\times (X_t^z\cap N_z^c)$ is a Borel rectangle and $W_z\cap(N_z^c\times N_z^c)$ is a countable union of $(X_t^z\cap N_z^c)\times (X_t^z\cap N_z^c)$. 
By Proposition~\ref{th:positive_contractive}  $\chi_{W_z}$ is a positive contractive Schur multiplier. 
Therefore $\chi_W$ is a positive contractive Schur $C_0(Z)$-multiplier by Theorem~\ref{th_poscentral}. 
\end{proof}

\subsection{Idempotent convolution multipliers}
\label{ssec:idempconv}

We next provide some 
examples of idempotent convolution multipliers. 
The following is immediate from Theorem \ref{conv_ab} and \cite[Theorem 2.1]{IlSp05}.

\begin{corollary}\label{co:idempconvcoset}
Suppose $G$ is an abelian locally compact group and $W\subseteq G\times \Gamma$ is a 
measurable set, such that $\chi_W \in \mathfrak{S}_\text{\rm conv}^\id$. 
Then $\| \chi_W \|_\Sch \leq 1$ if and only if $W$ is an open coset of $G \times \Gamma$ .
\end{corollary}

It is clear that if $G$ is abelian, and $C$ and $D$ are open cosets of $G$ and $\Gamma$ respectively, then $C \times D$ is an open coset of $G \times \Gamma$ and therefore $\chi_{C \times D}$ is an idempotent convolution multiplier of norm 1 by Corollary~\ref{co:idempconvcoset}. 
The following example shows that not all idempotent convolution multipliers of norm 1 are of this product form.

\begin{example}\label{ex:idempconvdiagonalsubgrp}
{
\rm
Consider the abelian group $G = \RR \times \ZZ_2$, and note that $G$ is isomorphic to its dual group $\Gamma$.
Define
\[
	H := \{ (a,0,b,0) ,\ (c,1,d,1) : a,b,c,d \in \RR \}.
\]
It is clear that $H$ is an open subgroup of $G \times \Gamma$, but $H$ cannot be written as a product of 
subgroups of $G$ and $\Gamma$. 
}	
\end{example}

\begin{remark} 
{\rm
Let $G$ be an abelian locally compact group; by Theorem~\ref{conv_ab}, 
a contractive idempotent Herz--Schur convolution multiplier, say $F$, corresponds to 
a characteristic function $\chi_W$, 
for an open coset $W \subseteq G \times \Gamma$. 
In the following, we show more precisely 
how the family $(F(r))_{r \in G} \subseteq \cbm(C^*_r(\Gamma))$ arises. 
Suppose that $W = xH$ for an open subgroup $H$ of 
$G \times \Gamma$ and $x \in G \times \Gamma$. Let $\nu$ be the 
representation of $G \times \Gamma$ on $\ell^2((G \times \Gamma )/ H)$, given by $\nu(z) \delta_{yH} := \delta_{zyH}$
($z,y \in G \times \Gamma$), $\{\delta_{yH}\}_y$ be the standard orthonormal basis in $\ell^2((G \times \Gamma )/ H)$), and write $\overline{\nu}$ for the unitary representation $\gamma \mapsto \nu(e , \gamma)$ of $\Gamma$.
For $r\in G$, let $u_r\in B(\Gamma)$ be the function given by
\[
    u_r : \Gamma \to \CC ;\ u_r(\gamma) := \ip{\overline{\nu}(\gamma) \delta_{(r,e)H}}{\delta_W} .
\]
Then
\[
\begin{split}
    S_{\chi_W} (\lambda^\Gamma_\gamma \otimes \lambda^G_r) &= \chi_W(r,\gamma)(\lambda^\Gamma_\gamma \otimes \lambda^G_r) = \ip{\delta_{(r,\gamma)H}}{\delta_W}(\lambda^\Gamma_\gamma \otimes \lambda^G_r)\\&=\ip{\nu(r , \gamma) \delta_H}{\delta_W} (\lambda^\Gamma_\gamma \otimes \lambda^G_r) \\
        &= \ip{\overline{\nu}(\gamma) \delta_{(r,e)H}}{\delta_W} (\lambda^\Gamma_\gamma \otimes \lambda^G_r) \\
        &= u_r(\gamma) \lambda^\Gamma_\gamma \otimes \lambda^G_r ,
\end{split}
\]
so the idempotent element $\chi_W \in B(G \times \Gamma)$ corresponds to the Herz--Schur convolution multiplier $F(r) := u_r$.
}
\end{remark}

It is immediate from Host's theorem that if 
$G$ is a connected locally compact group then $B(G)$ does not have non-trivial idempotent elements. 
We observe that this extends to idempotent convolution multipliers on abelian groups.
Indeed, let $\psi$ be an idempotent convolution multiplier of the dynamical system $(C^*_r(\Gamma) , G , \beta)$ 
and write $\psi = \chi_W$ for some $W \subseteq G\times \Gamma$.
For $x \in \Gamma$ and $s\in G$, let
\[
    W^x := \{ t\in G : (t,x)\in W \} \quad \text{ and } \quad W_s := \{ y\in \Gamma : (s,y)\in W \} .
\]

\begin{proposition}\label{pr:idempconvcosets}
Let $\psi=\chi_W\in \HSconv^\id(G)$ and $\|\psi\|_\Sch \leq 1$.
Then $W^x$ (resp. $W_s$) is an open coset of $G$ 
(resp. $\Gamma$) for all $x\in\Gamma$ (resp. $s\in G$).
\end{proposition}

\begin{proof}
Since for any $x\in \Gamma$, $s\in G$, we have $\psi^x=\chi_{W^x}$ and $\psi_s=\chi_{W_s}$,
the statement follows from  
\cite[Theorem 2.1]{IlSp05}, as $\psi^x\in B(G)$ and $\psi_s\in B(\Gamma)$.
\end{proof}

If $\psi=\chi_W\in \HSconv^\id(G)$ is contractive, 
as $\psi$ is separately continuous,
we obtain that $W_s=W_{s'}$ if $s$ and $s'$ are in the same connected component of $G$.
Similarly, we have $W^x=W^{x'}$ for $x,x'$ in the same connected component of $\Gamma$.
This implies the following corollary.

\begin{corollary}\label{co:idempconnected}
If the group $G$ (resp.\ $\Gamma$) is connected then any contractive 
idempotent multiplier $\psi\in \HSconv^\id(G)$ is given by $\psi=1\otimes\chi_A$ (resp.\ $\psi=\chi_A\otimes 1$), where $A$ is an open coset of $\Gamma$ (resp.\ $G$).
\end{corollary}

In particular, we have that 
$C^*_r(\RR)\rtimes_{\beta,r} \RR$ has no non-trivial idempotent Herz--Schur convolution multipliers, and any idempotent Herz--Schur convolution multiplier of $C(\mathbb T)\rtimes_{\beta ,r} \ZZ$ is given by $\chi_A\otimes 1$, where $A$ is a coset of $\ZZ$.

\begin{example}\rm
Let $G$ be a locally compact group.
Since $\Mcb L^1(G) = M(G)$, we have that $\gamma m_H \otimes \chi_C \in \Mcb (L^1(G) \projtens A(G))$, where $C$ is an open coset of $G$, $H$ is a compact subgroup and $\gamma$ is a character of $H$.  The corresponding convolution multiplier $\Lambda=(\mu_t)_{t\in G}$ 
is given by $\mu_t=\chi_C(t)\gamma m_H$. In fact, if $R$ is the completely bounded map 
\[
    R(f\otimes g) = \left((\gamma m_H)\ast f\right) \otimes\chi_C g,  \quad f\in L^1(G), g\in A(G),
\]
then
\[
    R^*(h\otimes\lambda_t) = \theta(\gamma m_H)(h)\otimes \chi_C\lambda_t = \theta(\gamma m_H) h \otimes \chi_C(t)\lambda_t.
\]
\end{example}

\begin{remark}\label{re:quantumgrpformulation}
{\rm 
For a (not necessarily abelian) locally compact group $G$ the algebra $C_0(G \times \hat G) := C_0(G)\otimes C_r^*(G)$ can be considered as a quantum group  with the comultiplication induced from comultiplications of the factors $C_0(G)$ and $C_r^*(G)$. In \cite{nsss} the authors give a characterisation of contractive idempotent functionals on $C^*$-quantum groups in terms of compact quantum subgroups and group-like unitaries of the subgroup. 
It would be interesting to use this characterisation to describe contractive convolution multipliers in the non-abelian case.
At present, however, a lack of examples of compact quantum subgroups of $C_0(G\times\hat G)$ 
impedes the application of the results of \cite{nsss} to convolution multipliers. 
}
\end{remark}

\medskip

\noindent {\bf Acknowledgements.}\ 
We would like to thank Przemyslaw Ohrysko for helpful conversations during the preparation of this paper. 
The second author was partially supported by the Ministry of Sciences of Iran and  School of Mathematics of Institute for research in fundamental sciences (IPM). Most of this work was completed when the second author was visiting Chalmers University of Technology. She would like to thank the Department of Mathematical Sciences of Chalmers University of Technology and the University of Gothenburg for warm hospitality. She would also like to thank Alireza Medghalchi and Massoud Amini for their support and encouragement during this work.

\bibliographystyle{plain}


\bibliography{specialcasesbib}

\begin{bibdiv}
\begin{biblist}

\bib{Arv74}{article}{
      author={Arveson, William},
       title={Operator algebras and invariant subspaces},
        date={1974},
        ISSN={0003-486X},
     journal={Ann. of Math. (2)},
      volume={100},
       pages={433\ndash 532},
         url={https://doi.org/10.2307/1970956},
      review={\MR{0365167}},
}

\bib{BeCo16}{article}{
      author={B\'{e}dos, Erik},
      author={Conti, Roberto},
       title={The {F}ourier--{S}tieltjes algebra of a {$C^*$}-dynamical
  system},
        date={2016},
        ISSN={0129-167X},
     journal={Internat. J. Math.},
      volume={27},
      number={6},
       pages={1650050, 50},
         url={https://doi.org/10.1142/S0129167X16500506},
      review={\MR{3516977}},
}

\bib{BS92}{article}{
      author={Blecher, David~P.},
      author={Smith, Roger~R.},
       title={The dual of the {H}aagerup tensor product},
        date={1992},
        ISSN={0024-6107},
     journal={J. London Math. Soc. (2)},
      volume={45},
      number={1},
       pages={126\ndash 144},
         url={https://doi.org/10.1112/jlms/s2-45.1.126},
      review={\MR{1157556}},
}

\bib{BoFe84}{article}{
      author={Bo\.{z}ejko, Marek},
      author={Fendler, Gero},
       title={Herz-{S}chur multipliers and completely bounded multipliers of
  the {F}ourier algebra of a locally compact group},
        date={1984},
     journal={Boll. Un. Mat. Ital. A (6)},
      volume={3},
      number={2},
       pages={297\ndash 302},
      review={\MR{753889}},
}

\bib{coh60}{article}{
      author={Cohen, Paul~J.},
       title={On a conjecture of {L}ittlewood and idempotent measures},
        date={1960},
        ISSN={0002-9327},
     journal={Amer. J. Math.},
      volume={82},
       pages={191\ndash 212},
         url={https://doi.org/10.2307/2372731},
      review={\MR{133397}},
}

\bib{CleMS}{unpublished}{
      author={Coine, Clement},
      author={Le~Merdy, Christian},
      author={Sukochev, Fedor},
       title={When do triple operator integrals take value in the trace
  class?},
        date={2019},
        note={To appear in Annals de l'institut Fourier. ArXiV: 1706.01662},
}

\bib{CaHa85}{article}{
      author={De~Canni\`ere, Jean},
      author={Haagerup, Uffe},
       title={Multipliers of the {F}ourier algebras of some simple {L}ie groups
  and their discrete subgroups},
        date={1985},
        ISSN={0002-9327},
     journal={Amer. J. Math.},
      volume={107},
      number={2},
       pages={455\ndash 500},
         url={https://doi.org/10.2307/2374423},
      review={\MR{784292}},
}

\bib{DR12}{article}{
      author={Dong, Zhe},
      author={Ruan, Zhong-Jin},
       title={A {H}ilbert module approach to the {H}aagerup property},
        date={2012},
        ISSN={0378-620X},
     journal={Integral Equations Operator Theory},
      volume={73},
      number={3},
       pages={431\ndash 454},
         url={https://doi.org/10.1007/s00020-012-1979-3},
      review={\MR{2945214}},
}

\bib{EfRu00}{book}{
      author={Effros, Edward~G.},
      author={Ruan, Zhong-Jin},
       title={Operator spaces},
      series={London Mathematical Society Monographs. New Series},
   publisher={The Clarendon Press, Oxford University Press, New York},
        date={2000},
      volume={23},
        ISBN={0-19-853482-5},
      review={\MR{1793753}},
}

\bib{EKS}{article}{
      author={Erdos, J.~A.},
      author={Katavolos, A.},
      author={Shulman, V.~S.},
       title={Rank one subspaces of bimodu\-les over maximal abelian
  selfadjoint algebras},
        date={1998},
        ISSN={0022-1236},
     journal={J. Funct. Anal.},
      volume={157},
      number={2},
       pages={554\ndash 587},
         url={https://doi.org/10.1006/jfan.1998.3274},
      review={\MR{1638277}},
}

\bib{eymard}{article}{
      author={Eymard, Pierre},
       title={L'alg\`ebre de {F}ourier d'un groupe localement compact},
        date={1964},
        ISSN={0037-9484},
     journal={Bull. Soc. Math. France},
      volume={92},
       pages={181\ndash 236},
         url={http://www.numdam.org/item?id=BSMF_1964__92__181_0},
      review={\MR{228628}},
}

\bib{Gr65}{article}{
      author={Greenleaf, Frederick~P.},
       title={Norm decreasing homomorphisms of group algebras},
        date={1965},
        ISSN={0030-8730},
     journal={Pacific J. Math.},
      volume={15},
       pages={1187\ndash 1219},
         url={http://projecteuclid.org/euclid.pjm/1102995276},
      review={\MR{0194911}},
}

\bib{Gro53}{article}{
      author={Grothendieck, Alexandre},
       title={R\'{e}sum\'{e} de la th\'{e}orie m\'{e}trique des produits
  tensoriels topologiques},
        date={1953},
        ISSN={0373-1375},
     journal={Bol. Soc. Mat. S\~{a}o Paulo},
      volume={8},
       pages={1\ndash 79},
      review={\MR{94682}},
}

\bib{Haagerup}{article}{
      author={Haagerup, Uffe},
       title={An example of a nonnuclear {$C^{\ast} $}-algebra, which has the
  metric approximation property},
        date={1978/79},
        ISSN={0020-9910},
     journal={Invent. Math.},
      volume={50},
      number={3},
       pages={279\ndash 293},
         url={https://doi.org/10.1007/BF01410082},
      review={\MR{520930}},
}

\bib{Ha80}{unpublished}{
      author={Haagerup, Uffe},
       title={Decomposition of completely bounded maps on operator algebras},
        date={1980},
        note={Unpublished manuscript.},
}

\bib{He74}{article}{
      author={Herz, Carl},
       title={Une g\'{e}n\'{e}ralisation de la notion de transform\'{e}e de
  {F}ourier-{S}tieltjes},
        date={1974},
        ISSN={0373-0956},
     journal={Ann. Inst. Fourier (Grenoble)},
      volume={24},
      number={3},
       pages={xiii, 145\ndash 157},
         url={http://www.numdam.org/item?id=AIF_1974__24_3_145_0},
      review={\MR{425511}},
}

\bib{Hos86}{article}{
      author={Host, Bernard},
       title={Le th\'{e}or\`eme des idempotents dans {$B(G)$}},
        date={1986},
        ISSN={0037-9484},
     journal={Bull. Soc. Math. France},
      volume={114},
      number={2},
       pages={215\ndash 223},
         url={http://www.numdam.org/item?id=BSMF_1986__114__215_0},
      review={\MR{860817}},
}

\bib{IlSp05}{article}{
      author={Ilie, Monica},
      author={Spronk, Nico},
       title={Completely bounded homomorphisms of the {F}ourier algebras},
        date={2005},
        ISSN={0022-1236},
     journal={J. Funct. Anal.},
      volume={225},
      number={2},
       pages={480\ndash 499},
         url={https://doi.org/10.1016/j.jfa.2004.11.011},
      review={\MR{2152508}},
}

\bib{jnr}{article}{
      author={Junge, Marius},
      author={Neufang, Matthias},
      author={Ruan, Zhong-Jin},
       title={A representation theorem for locally compact quantum groups},
        date={2009},
        ISSN={0129-167X},
     journal={Internat. J. Math.},
      volume={20},
      number={3},
       pages={377\ndash 400},
         url={https://doi.org/10.1142/S0129167X09005285},
      review={\MR{2500076}},
}

\bib{KaLa18}{book}{
      author={Kaniuth, Eberhard},
      author={Lau, Anthony To-Ming},
       title={Fourier and {F}ourier-{S}tieltjes algebras on locally compact
  groups},
      series={Mathematical Surveys and Monographs},
   publisher={American Mathematical Society, Providence, RI},
        date={2018},
      volume={231},
        ISBN={978-0-8218-5365-8},
      review={\MR{3821506}},
}

\bib{KP05}{article}{
      author={Katavolos, Aristides},
      author={Paulsen, Vern~I.},
       title={On the ranges of bimodule projections},
        date={2005},
        ISSN={0008-4395},
     journal={Canad. Math. Bull.},
      volume={48},
      number={1},
       pages={97\ndash 111},
         url={https://doi.org/10.4153/CMB-2005-009-4},
      review={\MR{2118767}},
}

\bib{kraus}{article}{
      author={Kraus, Jon},
       title={The slice map problem for {$\sigma $}-weakly closed subspaces of
  von {N}eumann algebras},
        date={1983},
        ISSN={0002-9947},
     journal={Trans. Amer. Math. Soc.},
      volume={279},
      number={1},
       pages={357\ndash 376},
         url={https://doi.org/10.2307/1999389},
      review={\MR{704620}},
}

\bib{lemerdy-tt}{article}{
      author={Le~Merdy, Christian},
      author={Todorov, Ivan~G.},
      author={Turowska, Lyudmila},
       title={Bilinear operator multipliers into the trace class},
        date={2020},
        ISSN={0022-1236},
     journal={J. Funct. Anal.},
      volume={279},
      number={7},
       pages={108649, 40},
         url={https://doi.org/10.1016/j.jfa.2020.108649},
      review={\MR{4107808}},
}

\bib{Liv95}{article}{
      author={Livshits, Leo},
       title={A note on {$0$}-{$1$} {S}chur multipliers},
        date={1995},
        ISSN={0024-3795},
     journal={Linear Algebra Appl.},
      volume={222},
       pages={15\ndash 22},
         url={https://doi.org/10.1016/0024-3795(93)00268-5},
      review={\MR{1332920}},
}

\bib{McKee}{unpublished}{
      author={McKee, Andrew},
       title={Weak amenability for dynamical systems},
        date={2020},
        note={Studia Math.},
}

\bib{mstt}{article}{
      author={McKee, Andrew},
      author={Skalski, Adam},
      author={Todorov, Ivan~G.},
      author={Turowska, Lyudmila},
       title={Positive {H}erz--{S}chur multipliers and approximation properties
  of crossed products},
        date={2018},
        ISSN={0305-0041},
     journal={Math. Proc. Cambridge Philos. Soc.},
      volume={165},
      number={3},
       pages={511\ndash 532},
         url={https://doi.org/10.1017/S0305004117000639},
      review={\MR{3860401}},
}

\bib{mtt}{article}{
      author={McKee, Andrew},
      author={Todorov, Ivan~G.},
      author={Turowska, Lyudmila},
       title={Herz--{S}chur multipliers of dynamical systems},
        date={2018},
        ISSN={0001-8708},
     journal={Adv. Math.},
      volume={331},
       pages={387\ndash 438},
         url={https://doi.org/10.1016/j.aim.2018.04.002},
      review={\MR{3804681}},
}

\bib{mck-t}{article}{
      author={McKee, Andrew},
      author={Turowska, Lyudmila},
       title={Exactness and soap of crossed products via herz-schur
  multipliers},
        date={2021},
     journal={J. Math. Anal. Appl.},
      volume={496},
      number={2},
       pages={124812},
      review={\MR{4189017}},
}

\bib{Neu00}{thesis}{
      author={Neufang, Matthias},
       title={Abstrakte harmonische analyse und modulhomomorphismen {\"u}ber
  von neumann-algebren},
        type={Ph.D. Thesis},
     address={Saarbr{\"u}cken, Germany},
        date={2000},
}

\bib{NRS08}{article}{
      author={Neufang, Matthias},
      author={Ruan, Zhong-Jin},
      author={Spronk, Nico},
       title={Completely isometric representations of {$M_{\rm cb}A(G)$} and
  {${\rm UCB}(\hat G)$}},
        date={2008},
        ISSN={0002-9947},
     journal={Trans. Amer. Math. Soc.},
      volume={360},
      number={3},
       pages={1133\ndash 1161},
         url={https://doi.org/10.1090/S0002-9947-07-03940-2},
      review={\MR{2357691}},
}

\bib{nsss}{article}{
      author={Neufang, Matthias},
      author={Salmi, Pekka},
      author={Skalski, Adam},
      author={Spronk, Nico},
       title={Contractive idempotents on locally compact quantum groups},
        date={2013},
        ISSN={0022-2518},
     journal={Indiana Univ. Math. J.},
      volume={62},
      number={6},
       pages={1983\ndash 2002},
         url={https://doi.org/10.1512/iumj.2013.62.5178},
      review={\MR{3205538}},
}

\bib{Oz00}{article}{
      author={Ozawa, Narutaka},
       title={Amenable actions and exactness for discrete groups},
        date={2000},
        ISSN={0764-4442},
     journal={C. R. Acad. Sci. Paris S\'{e}r. I Math.},
      volume={330},
      number={8},
       pages={691\ndash 695},
         url={https://doi.org/10.1016/S0764-4442(00)00248-2},
      review={\MR{1763912}},
}

\bib{Pe79}{book}{
      author={Pedersen, Gert~K.},
       title={{$C\sp{\ast} $}-algebras and their automorphism groups},
      series={London Mathematical Society Monographs},
   publisher={Academic Press, Inc. [Harcourt Brace Jovanovich, Publishers],
  London-New York},
        date={1979},
      volume={14},
        ISBN={0-12-549450-5},
      review={\MR{548006}},
}

\bib{Peller}{article}{
      author={Peller, Vladimir~V.},
       title={Hankel operators in the theory of perturbations of unitary and
  selfadjoint operators},
        date={1985},
     journal={Funktsional. Anal. i Prilozhen.},
      volume={19},
      number={2},
       pages={37\ndash 51,96},
}

\bib{Pisier-book}{book}{
      author={Pisier, Gilles},
       title={Introduction to operator space theory},
      series={London Mathematical Society Lecture Note Series},
   publisher={Cambridge University Press, Cambridge},
        date={2003},
      volume={294},
}

\bib{Ren80}{book}{
      author={Renault, Jean},
       title={A groupoid approach to {$C^{\ast} $}-algebras},
      series={Lecture Notes in Mathematics},
   publisher={Springer, Berlin},
        date={1980},
      volume={793},
        ISBN={3-540-09977-8},
      review={\MR{584266}},
}

\bib{Ren97}{article}{
      author={Renault, Jean},
       title={The {F}ourier algebra of a measured groupoid and its
  multipliers},
        date={1997},
        ISSN={0022-1236},
     journal={J. Funct. Anal.},
      volume={145},
      number={2},
       pages={455\ndash 490},
         url={https://doi.org/10.1006/jfan.1996.3039},
      review={\MR{1444088}},
}

\bib{STT11}{article}{
      author={Shulman, Victor~S.},
      author={Todorov, Ivan~G.},
      author={Turowska, Lyudmila},
       title={Closable multipliers},
        date={2011},
        ISSN={0378-620X},
     journal={Integral Equations Operator Theory},
      volume={69},
      number={1},
       pages={29\ndash 62},
         url={https://doi.org/10.1007/s00020-010-1819-2},
      review={\MR{2749447}},
}

\bib{ss}{article}{
      author={Sinclair, Allan~M.},
      author={Smith, Roger~R.},
       title={Factorization of completely bounded bilinear operators and
  injectivity},
        date={1998},
        ISSN={0022-1236},
     journal={J. Funct. Anal.},
      volume={157},
      number={1},
       pages={62\ndash 87},
         url={https://doi.org/10.1006/jfan.1998.3260},
      review={\MR{1637933}},
}

\bib{Stan09}{article}{
      author={Stan, Ana-Maria~Popa},
       title={On idempotents of completely bounded multipliers of the {F}ourier
  algebra {$A(G)$}},
        date={2009},
        ISSN={0022-2518},
     journal={Indiana Univ. Math. J.},
      volume={58},
      number={2},
       pages={523\ndash 535},
         url={https://doi.org/10.1512/iumj.2009.58.3452},
      review={\MR{2514379}},
}

\bib{Tak80}{book}{
      author={Takesaki, Masamichi},
       title={Theory of operator algebras},
   publisher={Springer-Verlag New York},
        date={1979},
      volume={I},
        ISBN={0387903917, 9780387903910},
}

\bib{Wi07}{book}{
      author={Williams, Dana~P.},
       title={Crossed products of {$C{^\ast}$}-algebras},
      series={Mathematical Surveys and Monographs},
   publisher={American Mathematical Society, Providence, RI},
        date={2007},
      volume={134},
        ISBN={978-0-8218-4242-3; 0-8218-4242-0},
         url={https://doi.org/10.1090/surv/134},
      review={\MR{2288954}},
}

\end{biblist}
\end{bibdiv}


\end{document}